\documentclass[12pt]{article}

\usepackage{graphicx}
\usepackage{amsmath}
\usepackage{amssymb}
\usepackage{latexsym}
\usepackage{subfigure}
\usepackage{crop}
\usepackage{algorithmic, algorithm}
\usepackage{multirow}
\usepackage[section,subsection,subsubsection]{placeins}
\usepackage[colorlinks = true, pdfstartview = FitV, linkcolor = blue, citecolor = blue, urlcolor = blue]{hyperref}
\usepackage{bm}
\usepackage{bbm}
\usepackage{enumerate}
\usepackage{framed} 
\usepackage[framed]{ntheorem}

\newcommand {\zz}  { {\bf z} }
\newcommand {\bgg}  { {\bf g} }

\newcommand {\xx}  { {\bf x} }
\renewcommand {\aa}  { {\bf a} }

\newcommand {\yy}  { {\bf y} }

\newcommand {\Ex} { {\mathbb E} }

\newcommand {\qq}  { {\bf q} }

\newcommand {\pp}  { {\bf p} }

\newcommand {\vv}  { {\bf v} }

\newcommand {\bb}  { {\bf b} }

\newcommand{\hf}{\frac12}

\newcommand{\defeq}{\mathrel{\mathop:}=}
\newcommand{\defeqr}{=\mathrel{\mathop:}}

\newcounter{comment}
\def\comment{\refstepcounter{comment}\textbf{Comment \arabic{comment}: }}

\theoremclass{Theorem}
\theoremstyle{break}
\newframedtheorem{theorem}{Theorem}
\newframedtheorem{corollary}{Corollary}
\newframedtheorem{lemma}{Lemma}

\newenvironment{proof}[1][Proof]{\begin{trivlist}
\item[\hskip \labelsep {\bfseries #1}]}{\end{trivlist}}

\newcommand{\qed}{\nobreak \ifvmode \relax \else
      \ifdim\lastskip<1.5em \hskip-\lastskip
      \hskip1.5em plus0em minus0.5em \fi \nobreak
      \vrule height0.75em width0.5em depth0.25em\fi}

\begin{document}

\pagestyle{plain} 
\setcounter{page}{1}
\title{Sub-Sampled Newton Methods II: Local Convergence Rates}
\author{Farbod Roosta-Khorasani\thanks{International Computer Science Institute, Berkeley, CA 94704
 and Department of Statistics, University of California at Berkeley, Berkeley, CA 94720. 
{\tt farbod/mmahoney@stat.berkeley.edu}.} \and Michael W. Mahoney\footnotemark[1]}
\maketitle
\begin{abstract}
Many data-fitting applications require the solution of an optimization problem involving a sum of large number of functions of high dimensional parameter. Here, we consider the problem of minimizing a sum of $n$ functions over a convex constraint set $\mathcal{X} \subseteq \mathbb{R}^{p}$ where both $n$ and $p$ are large. In such problems, sub-sampling as a way to reduce $n$ can offer great amount of computational efficiency.

Within the context of second order methods, we first give quantitative local convergence results for variants of Newton's method where the Hessian is uniformly sub-sampled.  Using random matrix concentration inequalities, one can sub-sample in a way that the curvature information is preserved.  Using such sub-sampling strategy, we establish locally Q-linear and Q-superlinear convergence rates. We also give additional convergence results for when the sub-sampled Hessian is regularized by modifying its spectrum or Levenberg-type regularization.

Finally, in addition to Hessian sub-sampling, we consider sub-sampling the gradient as way to further reduce the computational complexity per iteration. We use approximate matrix multiplication results from randomized numerical linear algebra (RandNLA) to obtain the
proper sampling strategy and we establish locally R-linear convergence rates. In such a setting, we also show that a very aggressive sample size increase results in a R-superlinearly convergent algorithm. 

While the sample size depends on the condition number of the problem, our convergence rates are problem-independent, i.e., they do not depend on the quantities related to the problem. Hence, our analysis here can be used to complement the results of our basic framework from the companion paper~\cite{romassn1} by exploring algorithmic trade-offs that are important in practice.


\end{abstract}

\section{Introduction}
\label{sec:intro}
Consider the optimization problem
\begin{equation}
\min _{\xx \in  \mathcal{D} \cap \mathcal{X}} F(\xx) = \frac{1}{n} \sum_{i=1}^{n} f_{i}(\xx),
\label{obj}
\end{equation}
where $\xx \in \mathbb{R}^{p}$, $\mathcal{X} \subseteq \mathbb{R}^{p}$ is a convex constraint set and $\mathcal{\mathcal{D}} = \bigcap_{i=1}^{n} \text{dom}(f_{i})$ is convex and open domain of $F$. Many data fitting applications can be expressed as~\eqref{obj} where each $f_{i}$ corresponds to an observation (or a measurement) which models the loss (or misfit) given a particular choice of the underlying parameter $\xx$. Examples of such optimization problems arise frequently in machine learning (e.g., logistic regression, support vector machines, neural networks and graphical models) and nonlinear inverse problems (e.g., PDE inverse problems). Many optimization algorithms have been developed to solve~\eqref{obj},~\cite{bertsekas1999nonlinear,nesterov2004introductory,boyd2004convex}. Here, we consider the regime where $n , p \gg 1$. In such high dimensional settings, the mere evaluation of the gradient or the Hessian can be computationally prohibitive. As a result, many of the classical \textit{deterministic} optimization algorithms might prove to be inefficient, if applicable at all. In this light, there has been a great deal of effort to design \textit{stochastic} variants which are efficient and can solve the modern ``big data'' problems without compromising much on the ``nice'' convergence behavior of the original deterministic counterpart. 

In this paper, we provide a detailed analysis of the use of sub-sampling as way to introduce randomness to the classical Newton's method and derive variants which are more suited for the modern big data problems. In doing so, we give conditions under which the local convergence properties of the full Newton's method is, to the extend possible, persevered. In particular, we will show that error recursions in all of our results exhibit a composite behavior whose dominating error term varies according to the distance of iterates to optimality. These results give a better control over various tradeoffs which exhibit themselves in different applications, e.g., practitioners  might require faster running-time while statisticians might be more interested in statistical aspects regarding the recovered solution. 
%
 
The rest of the paper is organized as follows: in Section~\ref{sec:general_background}, we first give a very brief background on the general methodology for optimizing~\eqref{obj}. The notation and the assumptions used in this paper are given in Sections~\ref{sec:notations} and~\ref{sec:assumtpions}. An overview of the main contributions of this paper is given in Section~\ref{sec:contributions}. In Section~\ref{sec:related_work}, we give a brief survey of the related work, and in their light, we discuss, in more details, the contributions of the present article. Section~\ref{sec:sub_hessian} addresses the local convergence behavior of sub-sampled Newton method in the case where only the Hessian is sub-sampled, while the gradient is used in full. Specifically, Section~\ref{sec:main_result_hess} gives such local results for when the sub-sampled Hessian is not ``altered'', whereas Section~\ref{sec:hessian_modify} establishes similar results for the cases where the sub-sampled Hessian is regularized by modifying its spectrum or Levenberg-type (henceforth called ridge-type) regularization. The case where the gradient, as well as Hessian, is sub-sampled is treated in Section~\ref{sec:subsampl_newton_hess_grad}. A few examples from generalized linear models (GLM) and support vector machine (SVM), very popular classes of problems in machine learning community, are given in Section~\ref{sec:examples}. Conclusions and further thoughts are gathered in Section~\ref{sec:conclusion}. All proofs are given in the appendix.

\subsection{General Background}
\label{sec:general_background}
For the sake of brevity and for this section only, lets assume $\mathcal{D} = \mathcal{X} = \mathbb{R}^{p}$. For optimizing~\eqref{obj}, the standard deterministic or full gradient method, which dates back to Cauchy~\cite{cauchy1847methode}, uses iterations of the form 
\begin{equation*}
\xx^{(k+1)} = \xx^{(k)} - \alpha_{k} \nabla F(\xx^{(k)}),
\end{equation*}
where $\alpha_{k}$ is the step size at iteration $k$. It is well-known,~\cite{nesterov2004introductory}, that full gradient method achieves rates of $\mathcal{O}(1/k)$ and $\mathcal{O}(\rho^{k}),\; \rho < 1$, for smooth and smooth-strongly convex objectives, respectively. However, when $n \gg 1$, the full gradient method can be inefficient because its iteration cost scales linearly in $n$. Consequently, stochastic variants of full gradient descent, e.g., (mini-batch) stochastic gradient descent (SGD) were developed~\cite{robbins1951stochastic,le2004large,li2014efficient,bertsekas1996neuro,bottou2010large,cotter2011better}. In such methods a subset $\mathcal{S} \subset \{1,2,\cdots,n\}$ is chosen at random and the update is obtained by
\begin{equation*}
\xx^{(k+1)} = \xx^{(k)} - \alpha_{k} \sum_{j \in \mathcal{S}} \nabla f_{j}(\xx^{(k)}).
\end{equation*}
When $|\mathcal{S}| \ll n$ (e.g., $|\mathcal{S}| = 1$ for simple SGD), the main advantage of these stochastic methods is that the iteration cost is independent of $n$ and can be much cheaper than the full gradient methods, making them suitable for modern problems with large $n$. However, this advantage comes at a cost: the convergence rate of these stochastic variants can be significantly slower that that of the full gradient descent. For example, under standard assumptions, it has been shown(e.g.,~\cite{nemirovski2009robust}) that, for a suitably chosen decreasing step-size sequence $\alpha_{k}$, the simple SGD iterations have an expected sub-optimality, i.e., $\mathcal{O}(1/\sqrt{k})$ and $\mathcal{O}(1/k)$ for smooth and smooth-strongly convex objectives, respectively. Since these sub-linear rates are slower than the corresponding rates for the full gradient method, great deal of efforts have been made to devise modifications to achieve the convergence rates of the full gradient methods while preserving the per-iteration cost of stochastic methods~\cite{roux2012stochastic,schmidt2013minimizing,johnson2013accelerating,shalev2013stochastic}.

The above class of methods are among what is known as \textit{first-order} methods where only the gradient information is used at every iteration. One attractive feature of such class of methods is their relatively low per-iteration-cost. However, despite such low cost, in almost all problems, incorporating curvature information (e.g., Hessian) as a form of scaling the gradient, i.e., 
\begin{equation*}
\xx^{(k+1)} = \xx^{(k)} - \alpha_{k} D_{k} \nabla F(\xx^{(k)}),
\end{equation*}
can significantly improve the convergence rate. Such class of methods which take the curvature information into account are known as \textit{second-order} methods, and compared to first-order methods, they enjoy superior convergence rate in both theory and practice. This is so since there is an implicit local \textit{scaling} of coordinates at a given $\xx \in \mathcal{D} \cap \mathcal{X}$, which is determined by the local curvature of $F$. This local curvature in fact determines the condition number of a $F$ at $\xx$. Consequently, by taking the curvature information into account (e.g., in the form of the Hessian), second order methods can rescale the gradient so it is a much more ``useful'' direction to follow. This is in contrast to first order methods which can only scale the gradient uniformly for all coordinates. Such second order information have long been used in many machine learning applications~\cite{bottou1998online,yu2010quasi,lin2008trust,martens2010deep,byrd2011use,byrd2012sample} as well as many PDE constrained inverse problems~\cite{doas12,rodoas1,rodoas2,bui2014solving,haber2000optimization}. 

It is well known,~\cite{nesterov2004introductory,boyd2004convex,nocedal2006numerical}, that the canonical example of second order methods, i.e., Newton's method, where $D_{k}$ is taken to be the inverse of the full Hessian and $\alpha_{k}=1$, i.e.,
\begin{equation*}
\xx^{(k+1)} = \xx^{(k)} - [\nabla^{2} F(\xx^{(k)})]^{-1} \nabla F(\xx^{(k)}),
\end{equation*}
converges at a quadratic rate for smooth-strongly convex objectives. Moreover, even for smooth convex functions, modifications of Newton's method has super-linear convergence. It is clear that both of these rates are much faster than those of the methods that are solely based on the gradient. It is also known that if the scaling matrix, $D_{k}$, is constructed to converge (in some sense) to the Hessian as $k \rightarrow \infty$, one can expect to obtain, asymptotically, at least a super-linear convergence rate~\cite{bertsekas1999nonlinear}. Another important property of Newton's method is \textit{scale invariance}. More precisely, for some new parametrization $\tilde{\xx} = A \xx$ for some invertible matrix A, the optimal search direction in the new coordinate system is $\tilde{\pp} = A \pp$ where $\pp$ is the original optimal search direction. By contrast, the search direction produced by gradient descent behaves in an opposite fashion as $\tilde{\pp} = A^{-T} \pp$. Such scale invariance property is important to more effectively optimize poorly scaled parameters; see~\cite{martens2010deep} for a very nice and intuitive explanation of this phenomenon. 

However when $n,p \gg 1$, the per-iteration-cost of such algorithm is significantly higher than that of first order methods. As a result, a line of research is to try to construct an approximation of the Hessian in a way that the update is computationally feasible, and yet, still provides sufficient second order information. One such class of methods are quasi-Newton methods, which are a generalization of the secant method to find the root of the first derivative for multidimensional problems. In such methods, the approximation to the Hessian is updated iteratively using only first order information from the gradients and the iterates through low-rank updates. Among these methods, the celebrated Broyden-Fletcher-Goldfarb-Shanno (BFGS) algorithm~\cite{nocedal2006numerical} and its limited memory version (L-BFGS)~\cite{nocedal1980updating,liu1989limited}, are the most popular and widely used members of this class. Another class of methods for approximating the Hessian is based on \textit{sub-sampling} where the Hessian of the full function $F$ is estimated using that of the randomly selected subset of functions $f_{i}$,~\cite{byrd2011use, byrd2012sample, erdogdu2015convergence,martens2010deep}. More precisely, a subset $\mathcal{S} \subset \{1,2,\cdots,n\}$ is chosen at random and, if the sub-sampled matrix is invertible, the update is obtained by 
\begin{subequations}
\begin{equation}
\xx^{(k+1)} = \xx^{(k)} - \alpha_{k} \big[\sum_{j \in \mathcal{S}} \nabla^{2} f_{j}(\xx^{(k)})\big]^{-1} \nabla F(\xx^{(k)}).
\label{sub_hess_unconstrained}
\end{equation}
In fact, sub-sampling can also be done for the gradient, obtaining a fully stochastic iteration
\begin{equation}
\xx^{(k+1)} = \xx^{(k)} - \alpha_{k} \big[\sum_{j \in \mathcal{S}_{H}} \nabla^{2} f_{j}(\xx^{(k)})\big]^{-1} \sum_{j \in \mathcal{S}_{\bgg}} \nabla f_{j}(\xx^{(k)}),
\label{sub_hess_grad_unconstrained}
\end{equation}
\label{sub_sampled_unconstrained}
\end{subequations}
where $\mathcal{S}_{\bgg}$ and $\mathcal{S}_{H}$ are sample sets used for approximating the Hessian and the gradient, respectively. The variants~\eqref{sub_sampled_unconstrained} are what we call \textit{sub-sampled Newton methods} in this paper. 

Unlike the sub-sampling in first order methods, theoretical properties of such techniques in the class of second-order methods are not yet very well understood. As a result, our aim in the present paper is to address the local convergence behavior of such sub-sampled Newton methods in a variety of situations (by local convergence, it is meant that the initial iterate is close enough to a local minimizer at which the sufficient conditions hold). This paper has an associated companion paper,~\cite{romassn1}, henceforth called SSN1, in which we consider globally convergent sub-sampled Newton algorithms and their convergence properties (by globally convergent algorithm, it is meant an algorithm that approaches the optimal solution starting from any initial point).  The reason for splitting into two papers is that, while SSN1~\cite{romassn1}, introduces several of the ideas in simpler settings, the more advanced techniques of this paper are needed to provide more control on the convergence behavior under a variety of assumptions.  We expect that the insight into the theoretical properties of the algorithms presented in this paper and SSN1~\cite{romassn1}, will enable the development of still-further improved sub-sampled Newton algorithms in a variety of applications in scientific computing, statistical data analysis, etc.

For the rest of this paper, we consider the general case of constrained optimization~\eqref{obj}. More specifically, given the current iterate, $\xx^{(k)} \in  \mathcal{D} \cap \mathcal{X}$, we consider the following iterative scheme, 
\begin{equation}
x_{k+1} = \arg \min _{\xx \in  \mathcal{D} \cap \mathcal{X}} \left\{ F(\xx^{(k)}) + (\xx-\xx^{(k)})^{T}\bgg(\xx^{(k)}) + \frac{1}{2 \alpha_{k}} (\xx-\xx^{(k)})^{T} H(\xx^{(k)})(\xx-\xx^{(k)})  \right\},
\label{structural_update_grad}
\end{equation}
where $\bgg(\xx^{(k)})$ and $H(\xx^{(k)})$ are some approximations to (in our case, sub-samples of) the actual gradient and the Hessian at the $k^{th}$ iteration, respectively. A variety of first and second order methods are of this form. For example, 
\begin{itemize}
    \item full Newton's method is obtained by setting $\bgg(\xx^{(k)}) = \nabla F(\xx^{(k)})$ and $H(\xx^{(k)}) = \nabla^{2}F(\xx^{(k)})$,
    \item the usual (projected) gradient descent is with the choice of $\bgg(\xx^{(k)}) = \nabla F(\xx^{(k)})$ and $H(\xx^{(k)}) = \mathbb{I}$,
		\item a step of the Frank-Wolfe,~\cite{jaggi2013revisiting}, algorithm can be obtained by considering $\bgg(\xx^{(k)}) = \nabla F(\xx^{(k)})$ and $H(\xx^{(k)}) =0$,
		\item stochastic (mini-batch) gradient descent is given by considering $\bgg(\xx^{(k)}) = 1/|\mathcal{S}_{\bgg}| \sum_{j \in \mathcal{S}_{\bgg}} \nabla f_{j}(\xx^{(k)})$ for some index set $\mathcal{S}_{g} \subseteq [n]$ and $H(\xx^{(k)}) = \mathbb{I}$,
		\item and finally, choosing the pair $\left\{ \bgg(\xx^{(k)}) = \nabla F(\xx^{(k)}), H(\xx^{(k)}) = 1/|\mathcal{S}_{H}| \sum_{j \in \mathcal{S}_{H}} \nabla^{2}f_{i}(\xx^{(k)}) \right\}$  or $\left\{\bgg(\xx^{(k)}) = 1/|\mathcal{S}_{\bgg}| \sum_{j \in \mathcal{S}_{\bgg}} \nabla f_{j}(\xx^{(k)}), H(\xx^{(k)}) = 1/|\mathcal{S}_{H}| \sum_{j \in \mathcal{S}_{H}} \nabla^{2}f_{i}(\xx^{(k)}) \right\}$ for some index sets $\mathcal{S}_{\bgg}, \mathcal{S}_{H} \subseteq [n]$ gives rise to sub-sampled Newton methods, which are the focus of this paper (as well as SSN1~\cite{romassn1}).
\end{itemize}
		

Unlike SSN1~\cite{romassn1} where our focus is mainly on designing sub-sampled Newton algorithms to guarantee global convergence, here, we concentrate on the actual speed of convergence. In particular, we aim to ensure that any such sub-sampled Newton algorithms preserves, at least locally, as much of the convergence properties of the full Newton's method as possible. In doing so, we need to ensure the following requirements:
\begin{enumerate}[(R.1)]
	\item \label{small_sample_size} Our sampling strategy needs to provide a sample size $|\mathcal{S}|$ which is independent of $n$, or at least smaller. Note that this is the same requirement as in SSN1~\cite[(R.1)]{romassn1}. However, as a result of the more intricate goals of the present paper, we will show that, comparatively, a larger sample size is required here than that used in SSN1~\cite{romassn1}.
	\item \label{preserve_spectrum} In addition, we need to ensure, at least probabilistically, that the sub-sampled matrix preserves the spectrum of the true Hessian as much as possible. If the gradient is also sub-sampled, we need to ensure that sampling is done in a way to keep as much of this first order information as possible. Note that unlike SSN1~\cite[(R.2)]{romassn1}, in order to preserve as much of the local convergence properties of the full Newton's method as possible, the mere invertibility of the sub-sampled Hessian is not enough.
	\item \label{local_rate} Finally, we need to ensure our algorithms enjoy a reasonably fast convergence rate which is, at least locally, similar to that of the full Newton's method. Note that unlike SSN1~\cite[(R.3)]{romassn1} where mere global convergence guarantee is required, here, the emphasis in on local convergence speed.
\end{enumerate}
In this paper, we address challenges~\hyperref[small_sample_size]{(R.1)},~\hyperref[preserve_spectrum]{(R.2)}, and~\hyperref[local_rate]{(R.3)}. More precisely, by using a random matrix concentration inequality as well as results from approximate matrix multiplication of randomized numerical linear algebra (RandNLA), we ensure~\hyperref[small_sample_size]{(R.1)} and~\hyperref[preserve_spectrum]{(R.2)}. To address~\hyperref[local_rate]{(R.3)}, we give algorithms whose local convergence rates can be made close to that of full Newton's method. These local rates coupled with the global convergence guarantees of SSN1~\cite{romassn1}, provide globally convergent  algorithms with \textit{fast} local rates (e.g., see SSN1~\cite[Theorems 2 and 7]{romassn1}). 
The present paper and the companion, SSN1~\cite{romassn1}, to the best of our knowledge, are the very first to thoroughly and quantitatively study the convergence behavior of such sub-sampled second order algorithms, in a variety of settings.

\subsection{Notation}
\label{sec:notations}
Throughout the paper, vectors are denoted by bold lowercase letters, e.g., $\vv$, and matrices or random variables are denoted by regular upper case letters, e.g., $V$, which is clear from the context. For a vector $\vv$, and a matrix $V$, $\|\vv\|$ and $\|V\|$ denote the vector $\ell_{2}$ norm and the matrix spectral norm, respectively, while $\|V\|_{F}$ is the matrix Frobenius norm. $\nabla f(\xx)$ and $\nabla^{2} f(\xx)$ are the gradient and the Hessian of $f$ at $\xx$, respectively. For two symmetric matrices $A$ and $B$, $A \succeq B$ indicates that $A-B$ is symmetric positive semi-definite. The superscript, e.g., $\xx^{(k)}$, denotes iteration counter and $\ln(x)$ denotes the natural logarithm of $x$. Throughout the paper, $\mathcal{S}$ denotes a collection of indices from $\{1,2,\cdots,n\}$, with potentially repeated items and its cardinality is denoted by $|\mathcal{S}|$. The tangent cone of the constraints at a (local) optimum $\xx^{*}$ is denoted by
\begin{equation}
\mathcal{K} \defeq \Big\{\pp \in \mathbb{R}^{p}; \; \exists t > 0 \text{ s.t. } \xx^{*} + t \pp \in \mathcal{D} \cap \mathcal{X} \Big\}.
\label{cone}
\end{equation}
For a vector $\vv$ and a matrix A, using such cone, we can define their $\mathcal{K}$-restricted norms, respectively, as 
\begin{subequations}
\label{norm_cone}
\begin{align}
\| \vv \|_{\mathcal{K}} &\defeq \max_{\pp \in \mathcal{K} \setminus \{0\}} \frac{|\pp^{T} \vv|}{\|\pp\|},
\label{norm_cone_vec} \\
\| A \|_{\mathcal{K}} &\defeq \max_{\pp, \qq \in \mathcal{K}\setminus \{0\}} \frac{|\pp^{T} A \qq|}{\|\pp\|\|\qq\|}.
\label{norm_cone_mat}
\end{align}
\end{subequations}
Similarly, one can define the $\mathcal{K}$-restricted maximum and the minimum eigenvalues of a symmetric matrix $A$ as
\begin{subequations}
\begin{align}
\lambda_{\min}^{\mathcal{K}}(A) \defeq \min_{\pp \in \mathcal{K} \setminus \{0\} } \frac{\pp^{T} A \pp}{\|\pp\|^{2}}, \label{lambda_min_cone}\\
\lambda_{\max}^{\mathcal{K}}(A) \defeq \max_{\pp \in \mathcal{K} \setminus \{0\} } \frac{\pp^{T} A \pp}{\|\pp\|^{2}}. \label{lambda_max_cone}
\end{align}
\label{lambda_cone}
\end{subequations}
Alternatively, let $U$ be an orthonormal basis for the cone $\mathcal{K}$. The definitions above are equivalent to the following:
\begin{align*}
\| \vv \|_{\mathcal{K}} &= \|U^{T} \vv\|, \\
\| A \|_{\mathcal{K}} &= \| U^{T} A U\|,  \\
\lambda_{\min}^{\mathcal{K}}(A) &= \lambda_{\min}(U^{T} A U),\\
\lambda_{\max}^{\mathcal{K}}(A) &= \lambda_{\max}(U^{T} A U),
\end{align*}
where $\lambda_{\max}(A)$ and $\lambda_{\min}(A)$ are, respectively, the usual maximum and minimum eigenvalues of $A$, i.e., computed with respect to all vectors in $\mathbb{R}^{p}$. This representation allows us to define any $\mathcal{K}$-restricted eigenvalue of $A$ as
\begin{equation}
\lambda_{i}^{\mathcal{K}}(A) = \lambda_{i} (U^{T} A U).
\label{lambda_i_cone}
\end{equation}

Throughout this paper we make use of two standard definitions of convergence rate: \textit{Q-convergence rate} and \textit{R-convergence rate}. Recall that a sequence of vectors $\{\zz^{(k)}\}_{k}$ is said to converge Q-linearly to a limiting value $\zz^{*}$, if for some $0 \leq \rho < 1$, 
\begin{equation*}
\lim \sup_{k} \frac{\|\zz^{(k+1)} - \zz^{*}\|}{\|\zz^{(k)} - \zz^{*}\|} = \rho.
\end{equation*}
Q-superlinear convergence is defined similarly as 
\begin{equation*}
\lim \sup_{k} \frac{\|\zz^{(k+1)} - \zz^{*}\|}{\|\zz^{(k)} - \zz^{*}\|}= 0.
\end{equation*}
The notion of R-convergence rate is an extension which captures sequences which still converge reasonably fast, but whose ``speed'' is variable. A sequence of vectors $\{\zz^{(k)}\}_{k}$ is said to converge R-superlinearly to a limiting value $\zz^{*}$, if $$\|\zz^{(k)} - \zz^{*}\| \leq R r^{(k)},$$ for some $R > 0$ and a sequence $\{r^{(k)}\}_{k}$ such that $$\lim\sup_{k} \frac{r^{(k+1)}}{r^{(k)}} = \rho < 1.$$
R-superlinear convergence is similarly defined by requiring that $$\lim\sup_{k} \frac{r^{(k+1)}}{r^{(k)}} = 0.$$

\subsection{Assumptions}
\label{sec:assumtpions}
\begin{table}[htb]
\centering
 \begin{tabular}{||c || c || c || c ||} 
 \hline
 \multicolumn{4}{|c|}{Assumptions for Convergence of Algorithms} \\
 \hline
 Algorithm & Hessian Lipschitz & Global Regularity & Local Regularity \\ [0.5ex] 
 \hline\hline
 \ref{alg1} & \eqref{F_Lip} & \eqref{strong_convex_boundedness} & ---  \\ 
 \ref{alg1} & \eqref{F_Lip} & --- & \eqref{strong_convex_boundedness_opt}  \\ 
 \ref{alg1_superlinear} & \eqref{F_Lip} & \eqref{strong_convex_boundedness} & --- \\
\ref{alg1_superlinear} & \eqref{F_Lip} & & \eqref{strong_convex_boundedness_opt}   \\
 \ref{alg_spectral} & \eqref{F_Lip_spectral} & \eqref{smoothness_spectral},~\eqref{strong_conv_spectral} &  ---\\
\ref{alg_spectral} & \eqref{F_Lip_spectral} & --- & \eqref{smoothness_spectral_opt},~\eqref{strong_conv_spectral_opt}  \\
 \ref{alg_ridge} & \eqref{F_Lip} & \eqref{strong_convex_boundedness} & --- \\
\ref{alg_ridge} & \eqref{F_Lip} &  & \eqref{strong_convex_boundedness_opt}   \\
 \ref{alg_hessian_grad} & \eqref{F_Lip} & \eqref{strong_convex_boundedness} &  ---\\
\ref{alg_hessian_grad} & \eqref{F_Lip} &  --- & \eqref{strong_convex_boundedness_opt}  \\
 \ref{alg_simult_hessian_grad} & \eqref{F_Lip} & --- & \eqref{strong_convex_boundedness_opt},~\eqref{grad_boundedness_opt}\\ [1ex] 
 \hline
 \end{tabular}
\caption{Summary of the assumptions used for convergence of different algorithms. Global regularity refers to smoothness and strong convexity of $F$ for $\forall \xx \in  \mathcal{D} \cap \mathcal{X}$, whereas local regularity refers to such properties for $F$ but only at a local optimum $\xx^{*}$. For each algorithm, we give separate convergence results under the global regularity as well as only the local regularity.  
\label{table_assumptions}}
\end{table}

Our theoretical analysis is based on the following assumptions. 
\begin{enumerate}
	\item \textbf{Lipschitz Hessian:} Throughout this article, we assume that each $f_{i}$ is twice-differentiable and has a Lipschitz continuous Hessian with respect to the cone $\mathcal{K}$, i.e., for some $L > 0$,
\begin{equation}
\|\nabla^{2} f_{i}(\xx) - \nabla^{2} f_{i}\big(\yy) \|_{\mathcal{K}} \leq L \|\xx - \yy\|, \;\text{ s.t. } \xx-\yy \in \mathcal{K} , \; i=1,2,\ldots,n,
\label{F_Lip}
\end{equation}
where $\| A \|_{\mathcal{K}}$ is defined in~\eqref{norm_cone_mat}.
	\item \textbf{Hessian Regularity:} Regularity of the Hessian refers to smoothness and strong convexity of $F$. Such properties can either be required for $\forall \xx \in \mathcal{D} \cap \mathcal{X}$ or alternatively, only at a local minimum, $\xx^{*}$,  which is assumed to always exist. In this paper, these properties are referred to as \textit{global} and \textit{local} regularity, respectively. In particular, for a given $\xx \in \mathcal{D} \cap \mathcal{X}$, let $K(\xx)$ and $\gamma(\xx)$ be such that
	\begin{subequations}
\begin{align}
\|\nabla^{2} f_{i}(\xx) \|_{\mathcal{K}} &\leq K(\xx), \quad i=1,2,\ldots,n, \label{boundedness_x}\\
\lambda_{\min}^{\mathcal{K}} \left ( \nabla^{2} F(\xx) \right) &\geq \gamma(\xx) \label{F_strong_x},
\end{align}
\label{strong_convex_boundedness_x}
\end{subequations}
where $\|A\|_{\mathcal{K}}$ and $\lambda_{\min}^{\mathcal{K}} (A)$ are defined in~\eqref{norm_cone_mat} and~\eqref{lambda_min_cone}, respectively.
We make the following regularity assumptions:
	\begin{enumerate}
		\item \textbf{Global Hessian Regularity:} Throughout Sections~\ref{sec:sub_hessian} and~\ref{sec:subsampl_newton_hess_grad}, every convergence result is first given for the case where we have global smoothness and strong convexity with respect to $\mathcal{K}$. More precisely, we assume that 
\begin{subequations}
\begin{align}
\sup_{\xx \in  \mathcal{D} \cap \mathcal{X}} K(\xx) &\defeqr  K < \infty, \label{boundedness}\\
 \inf_{\xx \in  \mathcal{D} \cap \mathcal{X}} \gamma(\xx) &\defeqr \gamma > 0, \label{F_strong}
\end{align}
\label{strong_convex_boundedness}
\end{subequations}
where $K(\xx)$ and $\gamma(\xx)$ are defined in~\eqref{strong_convex_boundedness_x}. Note that since $\mathcal{D} \cap \mathcal{X}$ is convex, Assumption~\eqref{F_strong} implies strong convexity of $F$ which, in turn, implies the \textit{uniqueness} of $\xx^{*}$. 

\item \textbf{Local Hessian Regularity:} Assumptions~\eqref{strong_convex_boundedness} are rather restrictive and limit the class of problems where the methods in the present paper apply. Fortunately, these assumptions can be further relaxed and be required to hold \textit{only} locally at a local optimum $\xx^{*}$. In other words, we require that   
\begin{subequations}
\begin{align}
K(\xx^{*}) &\defeqr K^{*} < \infty, \label{boundedness_opt}\\
\gamma(\xx^{*}) &\defeqr \gamma^{*} > 0, \label{F_strong_opt}
\end{align}
\label{strong_convex_boundedness_opt}
\end{subequations}
where $K(\xx^{*})$ and $\gamma(\xx^{*})$ are defined in~\eqref{strong_convex_boundedness_x}. Note that since $\mathcal{D} \cap \mathcal{X}$ is convex, Assumption~\eqref{F_strong_opt} imply that $\xx^{*}$ is an \textit{isolated} local optimum. These assumptions are relaxations of~\eqref{strong_convex_boundedness}. For example,~\eqref{F_strong_opt} appears as a second order sufficiency condition for a local optimum $\xx^{*}$.  In fact, some highly non-convex problems with multiple local minima exhibit such local structures, e.g.,~\cite{shamir2015fast}. As a result, each analysis in Sections~\ref{sec:sub_hessian} and~\ref{sec:subsampl_newton_hess_grad} is complemented with a convergence results using Assumptions~\eqref{strong_convex_boundedness_opt}, to make them applicable in a more general setting.

\end{enumerate}

\item \textbf{Locally Bounded Gradient:} In Section~\ref{sec:subsampl_newton_hess_grad}, in addition to the Hessian, the gradient is also sub-sampled. For some of the results presented there, we give convergence results using the following regularity assumptions on the gradients in the form of boundedness at a local optimum $\xx^{*}$:
\begin{equation}
\|\nabla f_{i}(\xx^{*})\|_{\mathcal{K}} \leq G^{*} < \infty, \quad i =1,2,\ldots,n.
\label{grad_boundedness_opt}
\end{equation} 
\end{enumerate}

For the results of Section~\ref{sec:spectral_reg} only, due to technical reasons which will be explained there, we need to remove $\mathcal{K}$-constrained condition from the above assumptions. More specifically,~\eqref{F_Lip} will be replaced by
\begin{subequations}
\begin{equation}
\|\nabla^{2} f_{i}(\xx) - \nabla^{2} f_{i}\big(\xx^{*}) \| \leq L \|\xx - \xx^{*}\|, \; \forall \xx  \in \mathcal{D} \cap \mathcal{X} , \; i=1,2,\ldots,n,
\label{F_Lip_spectral}
\end{equation}
and instead of~\eqref{strong_convex_boundedness}, we require that (for notational simplicity, we overload the constants $K$, $\gamma$, $K^{*}$, and $\gamma^{*}$)
\begin{align}
\sup_{\xx \in \mathcal{D} \cap \mathcal{X}} \|\nabla^{2}f_{i}(\xx) \| &\defeqr K < \infty, \quad i=1,2,\ldots,n, \label{smoothness_spectral} \\
\inf_{\xx \in \mathcal{D} \cap \mathcal{X}} \lambda_{\min} \left( F(\xx) \right) &\defeqr \gamma > 0, \label{strong_conv_spectral}
\end{align}
where $\|A\|$ and $\lambda_{\min}(A)$ are, respectively, the usual spectral norm and minimum eigenvalues of $A$, i.e., computed with respect to all vectors in $\mathbb{R}^{p}$. Assumption~\eqref{strong_convex_boundedness_opt} will similarly be replaced with 
\begin{align}
\| \nabla^{2}f_{i}(\xx^{*}) \| &\defeqr K^{*} < \infty, \quad i=1,2,\ldots,n, \label{smoothness_spectral_opt} \\
\lambda_{\min} \left( \nabla^{2}F(\xx^{*}) \right) &\defeqr \gamma^{*} > 0, \label{strong_conv_spectral_opt}
\end{align}
where $\xx^{*}$ is a local optimum. Note that in the unconstrained case or when $\mathcal{X}$ is open, since $\mathcal{K} = \mathbb{R}^{p}$, Assumptions~\eqref{spectral_assumtpions} and the ones restricted to the cone, $\mathcal{K}$, all coincide.
\label{spectral_assumtpions}
\end{subequations}

Table~\ref{table_assumptions} gives an overview of where each of the above assumptions are used to prove convergence of different algorithms presented in this paper. As noted before, for each algorithm, we first give convergence results using the global regularity assumptions, and then we provide separate convergence results by relaxing these assumptions to hold locally only at $\xx^{*}$.
  
\subsection{Contributions}
\label{sec:contributions}
\begin{table}[htb]
\centering
 \begin{tabular}{||c || c || c || c || c ||} 
 \hline
 \multicolumn{5}{|c|}{Summary of Results} \\
 \hline
 Theorem & Algorithm & Sub-Samp. & Local Conv. Rate & Hessian Reg.\\ [0.5ex] 
 \hline\hline
 \ref{uniform_newton_sufficient_cond} & \ref{alg1} & H & Q-Linear & --- \\ 
 \ref{uniform_newton_sufficient_cond_2},~\ref{uniform_newton_sufficient_cond_3} & \ref{alg1_superlinear} & H & Q-Superlinear & --- \\
 \ref{spectral_hessian_suff},~\ref{spectral_hessian_suff_relaxed}   &\ref{alg_spectral} & H & Q-Linear & Spectral \\
 \ref{ridge_hessian_suff} &\ref{alg_ridge} & H & Q-Linear & Ridge \\
 \ref{uniform_newton_grad_sufficient_cond} &\ref{alg_hessian_grad} & H \& G (indep.) & R-Linear & --- \\
 \ref{uniform_newton_grad_sufficient_cond_relax_2},~\ref{uniform_newton_grad_sufficient_cond_relax_3} &\ref{alg_simult_hessian_grad} & H \& G (simult.) & R-Linear \& R-Superlinear & --- \\ [1ex] 
 \hline
 \end{tabular}
\caption{Summary of the results. For the sub-sampling column, ``H'' denotes Hessian and ``G'' denotes gradient. For the same column, ``indep.'' refers to when sub-sampling the Hessian and the gradient is done independently of each other, while ``simult.'' refers to simultaneous sampling i.e., using one sample collection of indices for both. Hessian  regularization refers to Sections~\ref{sec:spectral_reg} and~\ref{sec:ridge}, where spectral or ridge-type regularization is used, respectively. 
\label{table_summary}}
\end{table}

The contributions of this paper can be summarized as follows:
\begin{enumerate}[(1)]
	\item Under the above assumptions, we study the local convergence behavior of various sub-sampled algorithms. We show that all of our error recursions exhibit a composite behavior whose dominating error term varies according to the distance of iterates to optimality. More specifically,  we will show that a quadratic error term dominates when we are far from an optimum and it transitions to lower degree terms according to the distance to optimality. 
	\item Our algorithms are designed for the following settings.
	\begin{enumerate}[(i)]
		\item Algorithms~\ref{alg1} and~\ref{alg1_superlinear} practically implement~\eqref{structural_update_grad} for the case where only the Hessian is sub-sampled, while the full gradient is used, i.e., $\bgg(\xx) = \nabla F(\xx)$. We give locally Q-linear and Q-superlinear convergence rates for Algorithms~\ref{alg1} and~\ref{alg1_superlinear}, respectively.
		\item Algorithms~\ref{alg_spectral} and~\ref{alg_ridge}, are modifications of Algorithm~\ref{alg1} in which the sub-sampled Hessian is regularized by modifying its spectrum or by ridge-type regularization, respectively. We show that such regularizations can be used to improve upon the initial progress of the algorithm, and for both of these regularization methods, we give Q-linear convergence rates.
		\item Algorithms~\ref{alg_hessian_grad} and~\ref{alg_simult_hessian_grad} are the implementation of the fully stochastic formulation of~\eqref{structural_update_grad}, in which the gradient as well the Hessian is sub-sampled. We show that one can sub-sample the gradient independently of the Hessian, or simultaneously using the same collection of sample indices and we provide R-linear and R-superlinear convergence rates for these algorithms.
	\end{enumerate}
	\item For all of these algorithms, we give quantitative convergence results, i.e., our bounds contain an actual worst-case convergence rates. Though the sample size depends on the condition number of the problem, we show that the rates for the (super)linear convergence phase are, in fact, problem-independent. As a result, by increasing the estimation accuracy, one can arbitrarily improve the speed of local convergence.
	\item Our results here only address the local convergence behaviors of different algorithms for when the initial iterate is sufficiently close to an optimum. However, under global regularity assumptions~\eqref{strong_convex_boundedness},  in SSN1~\cite{romassn1}, we show that by simple modifications of all of our algorithms here, one can obtain globally convergent algorithms which approach the optimum, regardless of the initial starting point. These connections guarantee that the modified algorithms converge globally with a local rate which is problem-independent; see SSN1~\cite[Theorems 2 and 7]{romassn1}.
\end{enumerate}
Table~\ref{table_summary} gives a summary of the main results of this paper. In addition, in Section~\ref{sec:related_work} and in light of the related work in the literature, we give more details of the contributions of the present paper.

\subsection{Related Work}
\label{sec:related_work}
The results of Section~\ref{sec:sub_hessian}, where the full gradient is used, offer computational efficiency for the regime where both $n$ and $p$ are large. However, it is required that $n$ is not so large as to make the gradient evaluation prohibitive. In such regime (where $n , p \gg 1$ but $n$ is not too large), similar results can be found in~\cite{byrd2011use,martens2010deep,friedlander2012hybrid, pilanci2015newton,erdogdu2015convergence}:

\begin{enumerate}[(1)]
	
	\item The pioneering work in~\cite{byrd2011use} establishes, for the first time, the convergence of Newton's method with sub-sampled Hessian and full gradient. There, two sub-sampled Hessian algorithms are proposed, where one is based on a matrix-free inexact Newton iteration and the other incorporates a preconditioned limited memory BFGS iteration. However, the results are asymptotic, i.e., for $k \rightarrow \infty$, and no quantitative convergence rate is given. In addition, convergence is established for the case where each $f_{i}$ is assumed to be strongly convex. Here, we concentrate on obtaining non-asymptotic and quantitative local convergence rates of such sub-sampled algorithm under milder assumption where only $F$ is required to be (locally) strongly convex. In addition, we extend these algorithms to the case where the gradient as well as the Hessian is sub-sampled. These results are presented in Sections~\ref{sec:sub_hessian} and~\ref{sec:subsampl_newton_hess_grad}.
	
	\item  Within the context of deep learning,~\cite{martens2010deep} is the first to study the application of a modification of Newton's method. It suggests a heuristic algorithm where at each iteration, the full Hessian is approximated by a relatively large subset of $\nabla^{2} f_i$'s, i.e. a “mini-batch”, and the size of such mini-batch grows as the optimization progresses. The resulting matrix is then damped in a Levenberg-Marquardt style,~\cite{levenberg1944algorithm,marquardt1963algorithm}, and conjugate gradient,~\cite{shewchuk1994introduction}, is used to approximately solve the resulting linear system
		
	\item Authors in~\cite{friedlander2012hybrid} were among the first to explore hybrid methods, combining the inexpensive iterations of incremental gradient algorithms and the steady convergence of full gradient methods, which exhibit the benefits of both approaches. Their analysis shows that by carefully increasing the sample size across iterations, it is possible to maintain the steady convergence rates of full-gradient methods. They also present a practical quasi-Newton implementation based on their approach. Such careful control over the increase of the sample size is one of the main ingredients of our analysis here.
	
	\item The work in~\cite{pilanci2015newton} is the first to use ``sketching'' within the context of Newton-like methods. The authors propose a randomized second-order method which is based on performing an approximate Newton's step using a randomly sketched Hessian. However, their algorithm is specialized to the cases where some square root of the Hessian matrix is readily available,i.e., some matrix $A(\xx) \in \mathbb{R}^{s \times p}$ with $s \geq p$, such that $\nabla^{2} F(\xx) = A^{T}(\xx) A(\xx)$. 

	\item Our analysis here most resembles that of~\cite{erdogdu2015convergence}. The work in~\cite{erdogdu2015convergence} is the first to establish quantitative convergence rate for the case where the Hessian is sub-sampled and the full gradient is used. The authors consider the \textit{two metric projection} formulation of the Newton's method and establish probabilistic linear convergence rate using Hoeffding matrix concentration result. They suggest an algorithm, where at each iteration, the spectrum of the sub-sampled Hessian is modified as a form of regularization. Our work here resembles that of~\cite{erdogdu2015convergence}, but is different in many aspects: 
	\begin{enumerate}[(i)]
		\item For the constrained optimization, the results in the present paper are given for Newton's method formulated as a \textit{scaled gradient projection method}~\eqref{structural_update_grad}, whereas the results in~\cite{erdogdu2015convergence} are given for what is called as two metric projection method. The main difficulty with the latter formulation is that an arbitrary positive definite matrix $H$ will not necessarily yield a descent direction, i.e., $F(\xx^{(k+1)}) > F(\xx^{(k)})$ for all $\alpha_{k} > 0$, or even the algorithm might not recognize, i.e.,  fail to stop at, an stationary point~\cite[Section 2.4]{bertsekas1999nonlinear}, which might result in instability in the algorithms relying on choosing the correct $H$. However, the two metric projection method, can potentially have an easier projection step as opposed to scaled gradient projection~\eqref{structural_update_grad}. In fact, the scaled gradient projection algorithm can be viewed as a form of projection method with norm $\|.\|_{H}$, resulting in an oblique projection. It should also be noted that, for the case of unconstrained optimization, the both formulations coincide.
		
		\item Within the context of scaled gradient projection method, we extend the work in~\cite{erdogdu2015convergence} in several dimensions. More specifically, we give a wide range of results from simple sub-sampling to adding different kinds of regularization. For each of these methods, we give sufficient conditions for local convergence of the respective algorithm, under both global and local regularity assumptions as in Section~\ref{sec:assumtpions}. We also give results for the case where both Hessian and the gradients are sub-sampled; see Sections~\ref{sec:sub_hessian} and~\ref{sec:subsampl_newton_hess_grad}.
		
		\item In situations where the Hessian is sparse, spectral regularization suggested in~\cite{erdogdu2015convergence} can destroy the sparsity. This in turn might make the storage and the computation with the resulting matrix less efficient. Here, in addition to extending the results using such spectral regularization, in Section~\ref{sec:ridge}, we suggest an alternative using ridge type regularization, which effectively achieves the same goal as the spectral counterpart, but at no extra computational or storage cots. 

			\item In addition to results with global regularity assumptions, in every section, we give additional convergence results under milder assumptions where the regularity is only required at a local minimum. In addition, most regularity assumptions are given with respect to the tangent cone of the constraints at a local minimum. These assumptions, depending on the geometry of this cone, can be significantly weaker than those made for unconstrained optimization (e.g., compare~\eqref{strong_convex_boundedness} and~\eqref{strong_convex_boundedness_opt} with~\eqref{spectral_assumtpions}).
	
		\item As for the sufficient condition to guarantee convergence using the spectral regularization, in~\cite[Corollary 3.5]{erdogdu2015convergence}, a sample size of $\Omega(\ln(p) \lambda^{2}K^{2}/\gamma^{6})$ is necessarily required (to at least obtain a positive region of convergence for the initial iterate), where $\lambda$ is the spectral regularization threshold as in Section~\ref{sec:spectral_reg}. However, the sample size required for Theorems~\ref{spectral_hessian_suff} and~\ref{spectral_hessian_suff_relaxed} is of order $\Omega(\ln(p) K^{2}/\gamma^{2})$, which is much smaller.
		
	\end{enumerate}
\end{enumerate} 

The results of Section~\ref{sec:subsampl_newton_hess_grad}, apply to more general settings where $n$ can be arbitrarily large. This is so since sub-sampling the gradient, in addition to that of the Hessian, allows for per-iteration cost which can be much smaller than $n$. Within the context of first order methods, there has been numerous variants of gradient sampling from a simple stochastic gradient descent,~\cite{robbins1951stochastic}, to the most recent improvements by incorporating the previous gradient directions in the current update~\cite{schmidt2013minimizing,senior2013empirical,bottou2010large,johnson2013accelerating}. For second order methods, such sub-sampling strategy has been successfully applied in large scale non-linear inverse problems~\cite{doas12,rodoas1,aravkin2012robust,van2013lost,haber2014simultaneous}. However, to the best of our knowledge, Section~\ref{sec:subsampl_newton_hess_grad} offers the first quantitative convergence results for such sub-sampled methods. In particular, these results guarantee the local convergence of many heuristic algorithms used in large scale nonlinear inverse problems, e.g.~\cite{doas12,rodoas1,rodoas2,roszas,haber2014simultaneous,aravkin2012robust,van2013lost}.

\section{Sub-Sampling Hessian}
\label{sec:sub_hessian}
For the optimization problem~\eqref{obj}, at each iteration, consider picking a sample of indices from $\{1,2,\ldots,n\}$, uniformly at random \textit{with} or \textit{without} replacement. Let $\mathcal{S}$ and $|\mathcal{S}|$ denote the sample collection and its cardinality, respectively and define 
\begin{equation}
H(\xx) \defeq \frac{1}{|\mathcal{S}|} \sum_{j \in \mathcal{S}} \nabla^{2} f_{j}(\xx),
\label{subsampled_H}
\end{equation}
to be the sub-sampled Hessian. 
As mentioned before in Section~\ref{sec:general_background}, in order for such Hessian sub-sampling to be effective in reducing $n$ as well as yielding a fast convergent algorithm, we need to ensure that the sample size $|\mathcal{S}|$ satisfies the requirement~\hyperref[small_sample_size]{(R.1)}, while as mentioned in~\hyperref[preserve_spectrum]{(R.2)}, the spectrum of $H(\xx)$ is as close to that of $\nabla^{2} F(\xx)$ as possible. The latter requirement is reinforced by noticing that the bounds in Section~\ref{sec:strutural_lemmas_hess} contain quantities which directly relate to the relative eigenvalue distributions of $H(\xx^{(k)})$ and $\nabla^{2}F(\xx^{(k)})$ (with respect to the cone $\mathcal{K}$). This indicates that the closer we can approximate the spectrum of the full Hessian, the faster the convergence is expected to be. Below, we make use of some matrix concentration inequalities to probabilistically guarantee the above properties.

For a given $\xx \in \mathcal{D} \cap \mathcal{X}$, let 
\begin{equation}
\kappa(\xx; \mathcal{K}) \defeq \frac{K(\xx)}{\gamma(\xx)},
\label{cond}
\end{equation}
be the condition number at $\xx$ but only with respect to vectors in $\mathcal{K}$, i.e., with $K(\xx)$ and $\gamma(\xx)$ defined as in~\eqref{strong_convex_boundedness_x}. Note that, depending on $\mathcal{K}$,~\eqref{cond} might be significantly smaller than the usual condition number, which is usually defined all vectors in $\mathbb{R}^{p}$. For example, consider the case where $\mathcal{D} = \mathbb{R}^{p}$ and $\mathcal{X} = \{\xx \in \mathbb{R}^{p}; \; A \xx = \bb \}$ for some full row-rank matrix $A \in \mathbb{R}^{m \times p}$ where $m < p$. Then $\mathcal{K}$ is the null space of $A$, i.e., $\mathcal{K} = \{\pp \in \mathbb{R}^{p}; \; A \pp = 0 \}$, and $\text{rank}(\mathcal{K}) = p-m$. As a result, one can compute $K(\xx)$ and $\gamma(\xx)$ as the largest and the smallest values among Rayleigh quotients of $\nabla^{2} f_{i}(\xx)$ and $\nabla^{2} F(\xx)$, respectively, but restricted to vectors in this $p-m$ dimensional sub-space. Depending on $A$, these values can be much smaller than the minimum and maximum of such Rayleigh quotients over the entire $\mathbb{R}^{p}$. Of course, in an unconstrained problem, $\kappa(\xx,\mathcal{K})$ coincides with the usual condition number at $\xx$.

With the above in mind, we can present our main sub-sampling Lemma to ensure that, to a desired accuracy, the spectrum of the full Hessian is preserved after sub-sampling.

\begin{lemma}[Uniform Hessian Sub-Sampling]
Given any $0 < \epsilon < 1$, $0 < \delta < 1$ and $\xx \in \mathcal{D} \cap \mathcal{X}$, if 
\begin{equation}
|\mathcal{S}| \geq \frac{16 \kappa^{2}(\xx,\mathcal{K})\ln(2p/\delta)}{\epsilon^{2}},
\label{uniform_sample_size_Hoeffding}
\end{equation}
then for $H(\xx)$ defined in~\eqref{subsampled_H}, we have
\begin{equation}
\Pr \Big( \left| \lambda_{i}^{\mathcal{K}} \left( \nabla^{2}F(\xx) \right) - \lambda_{i}^{\mathcal{K}}\left(H (\xx) \right) \right| \leq \epsilon \lambda_{i}^{\mathcal{K}} \left(\nabla^{2}F(\xx) \right); \; i = 1,2,\cdots,p \Big) \geq 1-\delta, 
\label{spectrum_preserve_epsilon}
\end{equation}
where $\kappa(\xx,\mathcal{K})$ is the $\mathcal{K}$-restricted condition number at $\xx$, defined in~\eqref{cond}, and $\lambda_{i}^{\mathcal{K}} (A)$ is defined in~\eqref{lambda_i_cone}.
\label{hoeffding_lemma}
\end{lemma}

Lemma~\ref{hoeffding_lemma} indicates that, with respect to the cone $\mathcal{K}$, it is indeed possible to sub-sample in a way that the spectrum of the sub-sampled Hessian does not deviate from that of the full Hessian.

Below, we show a couple of ways to improve upon the lower bound~\eqref{uniform_sample_size_Hoeffding}. First, we show that if each $f_{i}$ is convex, it is possible to reduce the constant factor further:
\begin{lemma}[Uniform Hessian Sub-Sampling: Improvement on the constant]
Given any $0 < \epsilon < 1$, $0 < \delta < 1$, and $\xx \in \mathcal{D} \cap \mathcal{X}$, if $\nabla^{2} f_{i}(\xx) \succeq 0, \forall i$ and
\begin{equation}
|\mathcal{S}| \geq \frac{4 \kappa^{2}(\xx,\mathcal{K})\ln(2p/\delta)}{\epsilon^{2}},
\label{uniform_sample_size_Hoeffding_const}
\end{equation}
then~\eqref{spectrum_preserve_epsilon} holds for $H(\xx)$ defined in~\eqref{subsampled_H}.
\label{hoeffding_lemma_const}
\end{lemma}

Even in the absence of convexity of each $f_{i}$, it is still possible to improve the dependence on $\ln(p)$ by using the concept of \textit{intrinsic dimension}~\cite[Chapter 7]{tropp2015introduction}. This, in turn, allows us to incorporate the dimension of the subspace $\mathcal{K}$ in our bound. More specifically, we have the following Lemma:
\begin{lemma}[Uniform Hessian Sub-Sampling: Improvement on $\ln(p)$]
Define 
\begin{equation*}
V \defeq \frac{|\mathcal{S}|}{n}  \sum_{i=1}^{n} (U^{T} \nabla^{2}f_{i}(\xx) U)^{2},
\end{equation*}
and let 
\begin{equation*}
d \defeq \frac{\text{trace}(V)}{\|V\|},
\end{equation*}
be the intrinsic dimension (a.k.a effective rank) of $V$. Given any $0 < \epsilon \leq 1/2$ , $0 < \delta < 1$, and $\xx \in \mathcal{D} \cap \mathcal{X}$, if
\begin{equation}
|\mathcal{S}| \geq \frac{16 \kappa^{2}(\xx,\mathcal{K})\ln(8d/\delta)}{3\epsilon^{2}},
\label{uniform_sample_size_Bernstein}
\end{equation}
and $\mathcal{S}$ is chosen uniformly at random with replacement, then~\eqref{spectrum_preserve_epsilon} holds for $H(\xx)$ defined in~\eqref{subsampled_H}.
\label{hoeffding_lemma_intrinsic}
\end{lemma}

\comment It is important to note that $d \leq \text{rank}(\mathcal{K}) \leq p$, and consequently, when $\mathcal{K}$ is low dimensional subspace,~\eqref{uniform_sample_size_Bernstein} can indeed offer some computational savings, compared to~\eqref{uniform_sample_size_Hoeffding}. In fact, if computing $d$ is not feasible, one can simply replace that with $\text{rank}(\mathcal{K})$ in~\eqref{uniform_sample_size_Bernstein}. For the rest of this paper, however, we will simply stick with the original bound~\eqref{uniform_sample_size_Hoeffding}, and note that extension to the above improvements is trivially done.

\comment It might be worth noting the differences between the above sub-sampling strategies and that of SSN1~\cite[Lemma 1]{romassn1}. These differences, in fact, boil down to the differences between the requirement~\hyperref[preserve_spectrum]{(R.2)} and the corresponding one in SSN1~\cite[Section 1.1, (R.2)]{romassn1}. As a result of a ``finer-grained'' analysis in the present paper and in order to preserve as much of the local convergence properties of the full Newton's method, Lemmas~\ref{hoeffding_lemma},~\ref{hoeffding_lemma_const}, and~\ref{hoeffding_lemma_intrinsic} require a larger sample size, i.e., in the order of $\kappa^{2}$ vs.\ $\kappa$ for SSN1~\cite[Lemma 1]{romassn1}, while delivering a much stronger guarantee about the spectrum of the sub-sampled Hessian. In contrast, for the global analysis in SSN1~\cite{romassn1}, we only need to ensure that the sub-sampled Hessian is invertible to yield a descent direction at every iteration. 

\comment \label{sample_size_alg} In our results and algorithms below, we use Lemma~\ref{hoeffding_lemma} with the following iteration independent sample sizes:
\begin{enumerate}[(i)]
	\item if Assumptions~\eqref{strong_convex_boundedness} hold, we use a larger and, yet, iteration-independent sample size
\begin{equation}
|\mathcal{S}| \geq \frac{16 K\ln(2p/\delta)}{\gamma\epsilon^{2}},
\label{uniform_sample_size_Hoeffding_alg}
\end{equation}
with $K$, $\gamma$ defined as in~\eqref{strong_convex_boundedness}, and
\item if Assumptions~\eqref{strong_convex_boundedness_opt} hold, we use
\begin{equation}
|\mathcal{S}| \geq \frac{16 K^{*}\ln(2p/\delta)}{\gamma^{*}\epsilon^{2}},
\label{uniform_sample_size_Hoeffding_alg_opt}
\end{equation}
with $K^{*}$, and $\gamma^{*}$ defined as in~\eqref{strong_convex_boundedness_opt}.
\end{enumerate}
For the results and algorithms of Section~\ref{sec:spectral_reg} where Assumptions~\eqref{spectral_assumtpions} are used, we use similar sample sizes, but with the respective constants defined as in~\eqref{spectral_assumtpions}.

\subsection{Main Results: Sub-Sampled Hessian}
\label{sec:main_result_hess}
Using the above sampling strategies, we can now present our main algorithms for the case of sub-sampling Hessian and full gradient~\eqref{structural_update}. Note that  for the following algorithms, we always consider the ``natural'' Newton step size, i.e., $\alpha_{k}=1$. A brief explanation for this choice is given in the beginning of Section~\ref{sec:main_proofs_hess}.

Throughout this section, for a given $\xx^{(k)} \in  \mathcal{D} \cap \mathcal{X}$, we consider the update
\begin{equation}
x_{k+1} = \arg\min _{\xx \in  \mathcal{D} \cap \mathcal{X}} \Big\{ F(\xx^{(k)}) + (\xx-\xx^{(k)})^{T}\nabla F(\xx^{(k)}) + \frac{1}{2 \alpha_{k}} (\xx-\xx^{(k)})^{T} H(\xx^{(k)})(\xx-\xx^{(k)})  \Big\},
\label{structural_update}
\end{equation}
where $H(\xx^{(k)})$ is as in~\eqref{subsampled_H}. Note that~\eqref{structural_update} is the same as~\eqref{structural_update_grad}  with $\bgg = \nabla F(\xx^{(k)})$, i.e., in~\eqref{structural_update} the full gradient in used.

\begin{algorithm}
\caption{Linearly Convergent Newton with Hessian Sub-Sampling}
\begin{algorithmic}[1]
\STATE \textbf{Input:} $\xx^{(0)}$, $0 < \delta < 1$, $0 < \epsilon < 1$
\STATE - Set the sample size, $|\mathcal{S}|$, with $\epsilon$ and $\delta$ as described in Comment~\ref{sample_size_alg}
\FOR {$k = 0,1,2, \cdots$ until termination} 
\STATE - Select a sample set, $\mathcal{S}$, of size $|\mathcal{S}|$ and $H(\xx^{(k)})$ as in~\eqref{subsampled_H}
\STATE - Update $\xx^{(k+1)}$ as in~\eqref{structural_update} with $H(\xx^{(k)})$ and $\alpha_{k}=1$
\ENDFOR
\end{algorithmic}
\label{alg1}
\end{algorithm}

\begin{theorem}[Error Recursion of~\eqref{structural_update}]
\label{uniform_newton_convergence}
\begin{enumerate}
	\item \textbf{Global Regularity:} Let Assumptions~\eqref{F_Lip} and~\eqref{strong_convex_boundedness}  hold and let $0 < \delta < 1$ and $0< \epsilon < 1$ be given. Set $|\mathcal{S}|$ as described in Comment~\ref{sample_size_alg}, and let $H(\xx^{(k)})$ be as in~\eqref{subsampled_H}. Then, for the update~\eqref{structural_update} with $\alpha_{k}=1$ , with probability $1-\delta$, we have
\begin{equation}
\|\xx^{(k+1)} - \xx^{*}\| \leq \rho_{0} \|\xx^{(k)} - \xx^{*}\| + \xi \|\xx^{(k)} - \xx^{*}\|^{2},
\label{convergence_subsampled_hessian}
\end{equation}
where 
\begin{equation}
\rho_{0} = \frac{\epsilon}{(1-\epsilon)}, \quad \text{and} \quad \xi = \frac{L}{2 (1-\epsilon) \gamma }.
\label{rho_xi}
\end{equation}
\item \textbf{Local Regularity:} Under Assumptions~\eqref{F_Lip},~\eqref{strong_convex_boundedness_opt}, and~\eqref{initial_cond_spec}, the result of theorem~\ref{uniform_newton_convergence} holds with 
\begin{equation}
\rho_{0} = \frac{2 \epsilon}{(1-\epsilon)},  \quad \text{and} \quad \xi = \frac{3 L}{(1-\epsilon) \gamma^{*} }.
\label{rho_xi_opt}
\end{equation}
\end{enumerate}
\end{theorem}

\comment Bounds given here exhibit a composite behavior where the error recursion, when far from an optimum, is first dominated by a quadratic term and it transforms to linear term near an optimum. What is rather interesting is that the rate for the linear phase is indeed \textit{independent} of any problem dependent quantities, and only depends on the sub-sampling accuracy! Of course, such problem dependent constants indeed appear in the lower bound for the sample size, in the form of (local) condition number.
 
Now we establish sufficient conditions for Q-linear convergence of sub-sampled Newton methods~\eqref{structural_update}. We remind that in this case, the sub-sampling is done only for the Hessian and the full gradient is used. 

\begin{theorem}[Q-Linear Convergence of Algorithm~\ref{alg1}]
\label{uniform_newton_sufficient_cond}
\begin{enumerate}
	\item \textbf{Global Regularity:}  Let Assumptions~\eqref{F_Lip} and~\eqref{strong_convex_boundedness}  hold and consider any $0 < \rho_{0} < \rho < 1$. Using Algorithm~\ref{alg1} with 
\begin{equation*}
\epsilon \leq \frac{\rho_{0} }{1 + \rho_{0}},
\end{equation*}
if
\begin{equation}
\|\xx^{(0)} - \xx^{*}\| \leq \frac{\rho-\rho_{0}}{\xi},
\label{initial_cond}
\end{equation}
with $\xi$ as in Theorem~\ref{uniform_newton_convergence}, we get locally Q-linear convergence
\begin{equation}
\|\xx^{(k)} - \xx^{*}\| \leq \rho \|\xx^{(k-1)} - \xx^{*}\|, \quad k = 1,\ldots, k_{0} 
\label{lin}
\end{equation} 
with probability $(1-\delta)^{k_{0}}$. 
	\item \textbf{Local Regularity:} Under Assumptions~\eqref{F_Lip},~\eqref{strong_convex_boundedness_opt}, Theorem~\ref{uniform_newton_sufficient_cond} holds with
\begin{equation*}
\epsilon \leq \frac{\rho_{0} }{2 + \rho_{0}}.
\end{equation*}
\end{enumerate}
\end{theorem}

In addition to Q-linear convergence, if the accuracy by which Hessian is sampled increases as the iterations progress, it is also possible to obtain Q-superlinear convergence. 

\begin{algorithm}
\caption{Superlinearly Convergent Newton with Hessian Sub-Sampling}
\begin{algorithmic}[1]
\STATE \textbf{Input:} $\xx^{(0)}$, $0 < \delta < 1$, $0 < \epsilon < 1$, $0 < \rho < 1$ 
\FOR {$k = 0,1,2, \cdots$ until termination} 
\STATE - Set $\epsilon^{(k)}$, for example as in Theorems~\ref{uniform_newton_sufficient_cond_2} or~\ref{uniform_newton_sufficient_cond_3}
\STATE - Set the sample size, $|\mathcal{S}^{(k)}|$, with $\epsilon^{(k)}$ and $\delta$ as described in Comment~\ref{sample_size_alg}
\STATE - Select a sample set, $\mathcal{S}^{(k)}$, of size $|\mathcal{S}^{(k)}|$ and $H(\xx^{(k)})$ as in~\eqref{subsampled_H}
\STATE - Update $\xx^{(k+1)}$ as in~\eqref{structural_update} with $H(\xx^{(k)})$ and $\alpha_{k}=1$
\ENDFOR
\end{algorithmic}
\label{alg1_superlinear}
\end{algorithm}

\begin{theorem}[Q-Superlinear Convergence of Algorithm~\ref{alg1_superlinear}: Geometric Growth]
\label{uniform_newton_sufficient_cond_2}
Let the respective assumptions of Theorem~\ref{uniform_newton_sufficient_cond} hold. Using Algorithm~\ref{alg1_superlinear}, with
\begin{equation*}
\epsilon^{(k)} = \rho^{k} \epsilon, \quad k = 0,1,\ldots,k_{0},
\end{equation*}
if $\xx^{(0)}$ satisfies~\eqref{initial_cond} with $\rho$, $\rho_{0}$, and $\xi^{(0)}$, we get locally Q-superlinear convergence
\begin{equation}
\|\xx^{(k)} - \xx^{*}\| \leq \rho^{k} \| \xx^{(k-1)} - \xx^{*}\|, \quad k = 1,\ldots, k_{0} 
\label{sup_lin}
\end{equation} 
with probability $(1-\delta)^{k_{0}}$, where $\xi^{(0)}$ is as in~\eqref{rho_xi} (or~\eqref{rho_xi_opt}) with $\epsilon^{(0)}$.
\end{theorem}

It is even possible to obtain Q-superlinear rate with a very mild growth of the sample size. For example, we can consider a growth which is only \textit{logarithmic} with the iteration counter.

\begin{theorem}[Q-Superlinear Convergence of Algorithm~\ref{alg1_superlinear}: Slow Growth]
\label{uniform_newton_sufficient_cond_3}
\begin{enumerate}
	\item \textbf{Global Regularity:} Let Assumptions~\eqref{F_Lip} and~\eqref{strong_convex_boundedness}  hold. Using Algorithm~\ref{alg1_superlinear} with 
\begin{equation*}
\epsilon^{(k)} = \frac{1}{1 + 2 \ln(4+k)}, \quad k = 0,1,\ldots, k_{0},
\end{equation*}
if $\xx^{(0)}$ satisfies
\begin{equation*}
\|\xx^{(0)} - \xx^{*}\| \leq \frac{2 \gamma }{\left(1 + 4\ln(2)\right) L},
\end{equation*}
we get locally Q-superlinear convergence
\begin{equation}
\|\xx^{(k)} - \xx^{*}\| \leq \frac{1}{\ln(3+k)} \| \xx^{(k-1)} - \xx^{*}\|, \quad k = 1,\ldots, k_{0} 
\label{sup_lin_slow}
\end{equation} 
with probability $(1-\delta)^{k_{0}}$.
\item \textbf{Local Regularity:} Under Assumptions~\eqref{F_Lip},~\eqref{strong_convex_boundedness_opt}, the result of theorem~\ref{uniform_newton_sufficient_cond_3} holds with 
\begin{equation*}
\epsilon^{(k)} = \frac{1}{1 + 4 \ln(4+k)}, \quad k = 0,1,\ldots, k_{0},
\end{equation*}
and $\xx^{(0)}$ which satisfies
\begin{equation*}
\|\xx^{(0)} - \xx^{*}\| \leq \frac{2 \gamma }{3 \left(1 + 8\ln(2)\right) L}.
\end{equation*}
\end{enumerate}
\end{theorem}

\comment Note that since $1/\ln(3 + k) \leq 1/2$ for all $k \geq 5$, Theorem~\ref{uniform_newton_sufficient_cond_3} guarantees that after only a few iterations, we see a very fast local convergence rate with only logarithmic increase in Hessian estimation accuracy.

\comment Theorems~\ref{uniform_newton_sufficient_cond_2} and~\ref{uniform_newton_sufficient_cond_3} state that in order to obtain locally superlinear convergence, as we get closer to the optimal solution, the Hessian is required to be estimated more accurately. The rate at which this estimation accuracy is improved, in turn, directly determines that of the Q-superlinear convergence.

\subsection{Modifying the Sample Hessian}
\label{sec:hessian_modify}
As was discussed earlier, the composite behavior in the error recursion of Theorem~\ref{uniform_newton_convergence} indicates that in early stages of the algorithm, when the iterates are far from an optimum, the error is dominated by a quadratic term. Now it can be seen from~\eqref{convergence_subsampled_hessian} that the quadratic term is negatively affected by small values of $\gamma$, and unlike the linear term, the estimation accuracy $\epsilon$ cannot reverse the effect. In other words, small values of $\gamma$ can hinder the initial progress and even using the full Hessian cannot address this issue. As a result, one might resort to regularization of the (estimated) Hessian to improve upon the initial convergence rate. Here, we explore two options for such regularization and will discuss the pros and cons of such strategies. 

\subsubsection{Spectral Regularization}
\label{sec:spectral_reg}
In this section, we follow the ideas presented in~\cite{erdogdu2015convergence}, by accounting for such a potentially negative factor $\gamma$, through a regularization of eigenvalue distribution of the sub-sampled Hessian. 
More specifically, for some $\lambda \geq 0 $, let
\begin{subequations}
\label{spectral_proj}
\begin{equation}
\hat{H} \defeq \mathcal{P}(\lambda; H),
\end{equation}
where $\mathcal{P}(\lambda; H)$ is an operator which is defined as
\begin{equation}
\mathcal{P}(\lambda; H) \defeq \lambda \mathbb{I} + \arg\max_{ X \succeq 0} \|H - \lambda \mathbb{I} - X\|_{F}.
\end{equation}
\end{subequations}
The operation~\eqref{spectral_proj} can be equivalently represented as 
\begin{equation*}
\mathcal{P}(\lambda; H) = \sum_{i=1}^{p} \max \left \{ \lambda_{i}(H), \lambda \right\} \vv_{i} \vv_{i}^{T},
\end{equation*}
with $\vv_{i}$ being the $i^{th}$ eigenvector of $H$ corresponding to the $i^{th}$ eigenvalue, $\lambda_{i}(H)$. 
Operation~\eqref{spectral_proj} can be performed using truncated SVD (TSVD). Note that although TSVD can be done through standard methods, faster randomized alternatives exist which provide accurate approximations to TSVD much more efficiently~\cite{halko2011finding}. 
%
%
In the constrained optimization case, where the restricted eigenvalues of $H$ with respect to the cone $\mathcal{K}$ are not a priori known, it is impossible to compute $\lambda_{i}^{\mathcal{K}}(H)$. As a result, we replace Assumptions~\eqref{F_Lip},~\eqref{strong_convex_boundedness} and~\eqref{strong_convex_boundedness_opt}, with the corresponding assumptions in~\eqref{spectral_assumtpions}.

For this regularization, we have the following error recursion result:
\begin{theorem}[Error Recursion of~\eqref{structural_update}: Spectral Reg.\ with Global Regularity]
\label{spectral_hessian_convergence}
Let Assumptions~\eqref{F_Lip_spectral},~\eqref{smoothness_spectral}, and~\eqref{strong_conv_spectral} hold and let $0 < \delta < 1$ and $0< \epsilon < 1$ be given. Set $|\mathcal{S}|$ as described in Comment~\ref{sample_size_alg}, and $H(\xx^{(k)})$ as in~\eqref{subsampled_H}. For some $\lambda > 0$, let 
\begin{equation*}
\hat{H}(\xx^{(k)}) = \mathcal{P} \Big(\lambda; H(\xx^{(k)}) \Big), 
\end{equation*}
where $\mathcal{P}(\lambda; . )$ is as in~\eqref{spectral_proj}.
Then, for the update~\eqref{structural_update} with $\alpha_{k} = 1$, if $\lambda \geq (1-\epsilon) \gamma$, it follows that~\eqref{convergence_subsampled_hessian} holds with probability $1-\delta$
where
\begin{equation*}
\rho_{0} = \frac{\lambda - (1-\epsilon) \gamma +  \gamma\epsilon}{\lambda}, \quad \text{and} \quad \xi = \frac{L}{2 \lambda }.
\end{equation*} 
\end{theorem}

Now we consider the local regularity counterparts of Assumptions~\eqref{smoothness_spectral} and~\eqref{strong_conv_spectral} by replacing them with~\eqref{smoothness_spectral_opt} and~\eqref{strong_conv_spectral_opt}. Under such relaxed assumptions, we have the following convergence result:

\begin{theorem}[Error Recursion of~\eqref{structural_update}: Spectral Reg.\ with Local Regularity]
\label{spectral_hessian_convergence_relaxed}
Under Assumptions~\eqref{F_Lip_spectral},~\eqref{smoothness_spectral_opt}, and~\eqref{strong_conv_spectral_opt} hold, if $\lambda \geq (1-\epsilon) \gamma^{*}$, the result of Theorem~\ref{spectral_hessian_convergence} holds with
\begin{equation*}
\rho_{0} = \frac{\lambda - (1-\epsilon) \gamma^{*} +  \gamma^{*} \epsilon}{\lambda}, \quad \text{and} \quad \xi = \frac{(\sqrt{p} + 1/2) L}{\lambda }.
\end{equation*} 
\end{theorem}

\comment The results of Theorems~\ref{spectral_hessian_convergence} and~\ref{spectral_hessian_convergence_relaxed} indicate that, by increasing the threshold $\lambda$, we get a better factor for the quadratic term, but a worse factor for the linear term. This is specially important in light of the fact that, without regularization and by Theorem~\ref{uniform_newton_convergence}, linear term only depends on sub-sampling accuracy which can be arbitrarily made small. Hence, it might be advisable to have larger threshold, $\lambda$, when far from the solution and gradually decrease the threshold, as iterates get closer to the optimum. This remedies the slow initial progress and yet allows for the sub-sampling dependent linear rate to be recovered when close enough to the solution.

\begin{algorithm}
\caption{Newton with Hessian Sub-Sampling and Spectral Regularization}
\begin{algorithmic}[1]
\STATE \textbf{Input:} $\xx^{(0)}$, $0 < \delta < 1$, $0 < \epsilon < 1$, and $0 < \epsilon_{0} < 1$ 
\STATE - Set the sample size, $|\mathcal{S}_{0}|$, with $\epsilon_{0}$ and $\delta$ as described in Comment~\ref{sample_size_alg}
\STATE - Set the sample size, $|\mathcal{S}|$, with $\epsilon$ and $\delta$ as described in Comment~\ref{sample_size_alg}
\FOR {$k = 0,1,2, \cdots$ until termination} 
\STATE - Select a sample set, $\mathcal{S}_{0}$, of size $|\mathcal{S}_{0}|$ and form $H_{0}(\xx^{(k)})$ as in~\eqref{subsampled_H}
\STATE - Compute $\lambda_{\min} \left(H_{0}(\xx^{(k)})\right)$ 
\STATE - Set the threshold, $\lambda^{(k)}$ as in Theorem~\ref{spectral_hessian_suff} (or Theorem~\ref{spectral_hessian_suff_relaxed})
\STATE - Select a sample set, $\mathcal{S}$, of size $|\mathcal{S}|$ and form $H(\xx^{(k)})$ as in~\eqref{subsampled_H}
\STATE - Form $\hat{H}(\xx^{(k)})$ as in~\eqref{spectral_proj} with $H(\xx^{(k)})$ and $\lambda^{(k)}$
\STATE - Update $\xx^{(k+1)}$ as in~\eqref{structural_update} with $\hat{H}(\xx^{(k)})$ and $\alpha_{k} = 1$
\ENDFOR
\end{algorithmic}
\label{alg_spectral}
\end{algorithm}

\begin{theorem}[Q-Linear Convergence of Algorithm~\ref{alg_spectral}: Global Regularity]
\label{spectral_hessian_suff}
Let Assumptions~\eqref{F_Lip_spectral},~\eqref{smoothness_spectral}, and~\eqref{strong_conv_spectral} hold. Using Algorithm~\ref{alg_spectral}, if 
\begin{align}
\|\xx^{(0)} - \xx^{*}\| &< \frac{\gamma}{3 L}, \label{initial_cond_spectral} \\
\epsilon &\leq \frac{1 }{6}, \label{epsilon_spectral} 
\end{align}
and at every iteration the threshold is chosen as 
\begin{equation}
\lambda^{(k)} \geq \left(\frac{1-\epsilon}{1-\epsilon_{0}}\right) \lambda_{\min} \left(H_{0}(\xx^{(k)})\right),
\label{spectral_threshold}
\end{equation}
we get locally Q-linear convergence with variable rates
\begin{equation*}
\|\xx^{(k)} - \xx^{*}\| \leq \rho^{(k-1)} \|\xx^{(k-1)} - \xx^{*}\|, \quad k = 1,\ldots, k_{0} 
\end{equation*} 
with probability $(1-\delta)^{2 k_{0}}$, where
\begin{equation*}
\rho^{(k)} = 1 - \frac{\gamma}{2 \lambda^{(k)}}.
\end{equation*}
\end{theorem}

\begin{theorem}[Q-Linear Convergence of Algorithm~\ref{alg_spectral}: Local Regularity]
\label{spectral_hessian_suff_relaxed}
Let Assumptions~\eqref{F_Lip_spectral},~\eqref{smoothness_spectral_opt}, and~\eqref{strong_conv_spectral_opt} hold. Using Algorithm~\ref{alg_spectral}, if 
\begin{align}
\|\xx^{(0)} - \xx^{*}\| &< \frac{\gamma^{*}}{6 (\sqrt{p} + 1/2)L}, \label{initial_cond_spectral_relax} \\
\epsilon &\leq \frac{1}{6}, \label{epsilon_spectral_relax} 
\end{align}
and at every iteration the threshold is chosen as 
\begin{equation}
\lambda^{(k)} > \left(\frac{(6 \sqrt{p} + 3)(1-\epsilon)}{(6 \sqrt{p} + 2)(1-\epsilon_{0})} \right) \lambda_{\min} \left(H_{0}(\xx^{(k)})\right),
\label{spectral_threshold_relaxed}
\end{equation}
the result of Theorem~\ref{spectral_hessian_suff} holds with
\begin{equation*}
\rho^{(k)} = 1 - \frac{\gamma^{*}}{2 \lambda^{(k)}}.
\end{equation*}
\end{theorem}

\comment It is easy to see that as $\lambda$ gets larger, by Theorems~\ref{spectral_hessian_suff} and~\ref{spectral_hessian_suff_relaxed}, the algorithm behaves more like first-order gradient descent methods. In fact, in the case where $\lambda = K$, the algorithm~\eqref{structural_update} indeed is gradient descent with $H(\xx^{(k)}) = \mathbb{I}$ and the constant step size $\alpha_{k} = 1/K$. In such a case, Theorems~\ref{spectral_hessian_suff} and~\ref{spectral_hessian_suff_relaxed} give similar results as linear convergence of gradient descent method for (locally) smooth and strongly convex functions~\cite[Theorem 2.1.15]{nesterov2004introductory}.

\comment To get convergence, we only need a rough estimate of the smallest eigenvalue of $\nabla^{2}F(\xx^{(k)})$ at iteration $k$.  As a result, one can consider relatively large $\epsilon_{0}$ in Algorithm~\ref{alg_spectral} and have a small sample size for such rough estimation. 

\comment The results of Theorems~\ref{spectral_hessian_suff} and~\ref{spectral_hessian_suff_relaxed} indicate that, when close enough to the solution, increasing the threshold $\lambda$ slows down the convergence rate! Hence, more aggressive regularization might, in fact, adversarially affect the efficiency of the algorithm. As a result, it is only advisable to have large thresholds at early stages of the algorithm. For example, using Theorem~\ref{spectral_hessian_convergence} and any given $0< \xi_{0} < 1$, it can be shown that if $\lambda \geq \beta L/(2 \xi_{0})$, for some $\beta > 1$, then while 
\begin{equation*}
\|\xx^{(k)} - \xx^{*}\| \geq \frac{\beta L - 2\gamma \xi_{0} + 4\gamma\xi_{0} \epsilon }{(\beta - 1) L \xi_{0}},
\end{equation*}
we have quadratic convergence, i.e.,
\begin{equation*}
\|\xx^{(k+1)} - \xx^{*}\| \leq \xi_{0} \|\xx^{(k)} - \xx^{*}\|^{2}.
\end{equation*}

\subsubsection{Ridge Regularization}
\label{sec:ridge}
As an alternative to the spectral regularization, we can consider the following simple ridge-type regularization 
\begin{equation}
\hat{H}(\xx) \defeq H(\xx) + \lambda \mathbb{I},
\label{ridge}
\end{equation}
for some $\lambda \geq 0$, similar to  the Levenberg-Marquardt type algorithms~\cite{levenberg1944algorithm,marquardt1963algorithm}. Such regularization might be preferable to the spectral regularization of Section~\ref{sec:spectral_reg}, as it avoids the projection operation~\eqref{spectral_proj} at every iteration. In addition, in order to establish convergence in Section~\ref{sec:spectral_reg}, the values of the threshold have to be chosen with regards to the eigenvalue distributions of $\nabla^{2} F(\xx^{k})$ (or $\nabla^{2}F(\xx^{*})$). As a result of this undesirable coupling, in Section~\ref{sec:spectral_reg}, assumptions on the full eigenvalue distribution (as opposed to the ones restricted to the cone of the constraints) were unavoidable, and depending on the cone $\mathcal{K}$,  this can be significant (compare Assumptions~\eqref{spectral_assumtpions} with their respective counterparts in~\eqref{F_Lip},~\eqref{strong_convex_boundedness} and~\eqref{strong_convex_boundedness_opt}). However, here $\lambda$ can be chosen a priori without any restriction and also, if desired, can be kept fixed across all iterations.  As a result, we can continue to use Assumptions~\eqref{F_Lip},~\eqref{strong_convex_boundedness} and~\eqref{strong_convex_boundedness_opt}, and still acquire convergence. For such regularization, we have the following results:

\begin{algorithm}
\caption{Newton with Hessian Sub-Sampling and Ridge Regularization}
\begin{algorithmic}[1]
\STATE \textbf{Input:} $\xx^{(0)}$, $0 < \delta < 1$, $0 < \epsilon < 1$, $\lambda \geq 0$
\STATE - Set the sample size, $|\mathcal{S}|$, with $\epsilon$ and $\delta$ as described in Comment~\ref{sample_size_alg}
\FOR {$k = 0,1,2, \cdots$ until termination} 
\STATE - Select a sample set, $\mathcal{S}$, of size $|\mathcal{S}|$ and form $H(\xx^{(k)})$ as in~\eqref{subsampled_H}
\STATE - Form $\hat{H}(\xx^{(k)})$ as in~\eqref{ridge} with $H(\xx^{(k)})$ and $\lambda$
\STATE - Update $\xx^{(k+1)}$ as in~\eqref{structural_update} with $\hat{H}(\xx^{(k)})$ and $\alpha_{k} = 1$
\ENDFOR
\end{algorithmic}
\label{alg_ridge}
\end{algorithm}

\begin{theorem}[Error Recursion of~\eqref{structural_update}: Ridge Regularization]
\label{ridge_hessian_convergence}
\begin{enumerate}
	\item \textbf{Global Regularity:} Let Assumptions~\eqref{F_Lip} and~\eqref{strong_convex_boundedness}  hold, and let $0 < \delta < 1$ and $0< \epsilon < 1$ be given. Set $|\mathcal{S}|$ as described in Comment~\ref{sample_size_alg} and $H(\xx^{(k)})$ as in~\eqref{subsampled_H}. For some $\lambda \geq 0$, let $\hat{H}(\xx^{(k)})$ be as in~\eqref{ridge}. Then, for the update~\eqref{structural_update} with $\alpha_{k} = 1$, it follows that~\eqref{convergence_subsampled_hessian} holds with probability $1-\delta$
where
\begin{equation*}
\rho_{0} = \frac{\lambda +  \gamma	 \epsilon}{(1-\epsilon) \gamma + \lambda}, \quad
\xi = \frac{L}{2(1-\epsilon) \gamma + 2\lambda}.
\end{equation*}
\item \textbf{Local Regularity:} Under Assumptions~\eqref{F_Lip},~\eqref{strong_convex_boundedness_opt}, and~\eqref{initial_cond_spec}
the result of theorem~\ref{ridge_hessian_convergence} holds with 
\begin{equation*}
\rho_{0} = \frac{2 \lambda +  2 \gamma^{*} \epsilon}{(1-\epsilon) \gamma^{*} + 2\lambda}, \quad
\xi = \frac{3L}{(1-\epsilon) \gamma^{*} + 2\lambda}.
\end{equation*}
\end{enumerate}
\end{theorem}

\comment From Theorem~\ref{ridge_hessian_convergence} it can be seen that increasing $\lambda$ results in decreasing the rate for the quadratic term, whilst increasing that of the linear term. As a result, there is a trade-off in choosing the values of regularization and it might be preferable to have larger values for $\lambda$ when far from the solution and gradually decrease it, as iterates get closer to the optimum. This is so because as $\lambda$ gets larger, the method tends to behave more like first order gradient descent.

\begin{theorem}[Q-Linear Convergence of Algorithm~\ref{alg_ridge}]
\label{ridge_hessian_suff}
\begin{enumerate}
	\item \textbf{Global Regularity:} Let Assumptions~\eqref{F_Lip} and~\eqref{strong_convex_boundedness}  hold. For any $\lambda \geq 0$, consider $\rho_{0}$ and $\rho$ such that
\begin{equation*}
1 - \frac{\gamma}{\gamma + \lambda} < \rho_{0} < \rho < 1.
\end{equation*}
Using Algorithm~\ref{alg_ridge} with
\begin{equation*}
\epsilon \leq \frac{\rho_{0} \gamma + (\rho_{0}  - 1) \lambda}{(1 + \rho_{0}) \gamma}, 
\end{equation*}
if~\eqref{initial_cond} holds with $\xi$ as in Theorem~\ref{ridge_hessian_convergence}, then with probability $(1-\delta)^{k_{0}}$, we get locally Q-linear convergence as in~\eqref{lin} with the rate $\rho$.
	\item \textbf{Local Regularity:} Under Assumptions~\eqref{F_Lip},~\eqref{strong_convex_boundedness_opt}, for any
\begin{equation*}
1 - \frac{\gamma^{*}}{\gamma^{*} + 2\lambda} < \rho_{0} < \rho < 1,
\end{equation*}
Theorem~\ref{ridge_hessian_suff} holds with 
\begin{equation*}
\epsilon \leq \frac{\rho_{0} \gamma^{*} + 2(\rho_{0}  - 1) \lambda }{(2 + \rho_{0}) \gamma^{*}}.
\end{equation*}
\end{enumerate}
\end{theorem}

\comment Similar observations as in Theorem~\ref{spectral_hessian_suff} and~\ref{spectral_hessian_suff_relaxed} can also be made here. In other words, we see that large value for $\lambda$ negatively affects the linear convergence phase, while it is beneficial for the early stages of iteration where the rate is quadratic. For example, using Theorem~\ref{ridge_hessian_convergence} and any given $0< \xi_{0} < 1$, it can be shown that if, for some $\beta > 1$,
\begin{equation*}
\lambda \geq \frac{\beta L - (1-\epsilon)2\gamma \xi_{0}}{2\xi_{0}},
\end{equation*}
then while 
\begin{equation*}
\|\xx^{(k)} - \xx^{*}\| \geq \frac{\beta L - 2\gamma \xi_{0} + 4\gamma\xi_{0} \epsilon}{(\beta - 1)L \xi_{0}},
\end{equation*}
we have quadratic convergence, i.e.,
\begin{equation*}
\|\xx^{(k+1)} - \xx^{*}\| \leq \xi_{0} \|\xx^{(k)} - \xx^{*}\|^{2}.
\end{equation*}

\section{Sub-Sampling Hessian \& Gradient}
\label{sec:subsampl_newton_hess_grad}
In order to compute the update $\xx^{(k+1)}$ in Section~\ref{sec:sub_hessian}, full gradient was used. In many problems, this can be a major bottleneck and reduction in computational costs can be made by considering sub-sampling the gradient as well. Such methods have recently been used in PDE constrained inverse problems,~\cite{rodoas1,rodoas2,roszas,doas12,haber2012effective,aravkin2012robust}, where each function $f_{i}$ is only implicitly available. In such problems, computing $\nabla f_{i}$ requires solving a PDE twice, amounting to a total cost of $2n$ PDE solves for each full gradient evaluation. In high dimensional settings where $n ,p \gg 1$, this can pose a significant challenge and sub-sampling the gradient can, at times, drastically reduce the computational complexity of many problems. As such, for this section, we consider the general update of~\eqref{structural_update_grad}, where $\bgg(\xx^{(k)})$ is a sub-sampled approximation to $\nabla F(\xx^{(k)})$.

As in Section~\ref{sec:sub_hessian}, consider picking a sample of indices from $\{1,2,\ldots,n\}$, uniformly at random with replacement. Also let $\mathcal{S}$ and $|\mathcal{S}|$ denote the sample collection and its cardinality, respectively. Let 
\begin{equation}
\bgg(\xx) \defeq \frac{1}{|\mathcal{S}|} \sum_{j \in \mathcal{S}} \nabla f_{j}(\xx),
\label{subsampled_G}
\end{equation}
be the sub-samples gradient. As mentioned before in the requirement~\hyperref[preserve_spectrum]{(R.2)}, we need to ensure that sampling is done in a way to keep as much of the first order information from the full gradient as possible. 

By a simple observation, the gradient $\nabla F(\xx)$ can be written in \textit{matrix-matrix product} from as 
\begin{align*}
\nabla F(\xx) = \begin{pmatrix}
\mid & \mid & & \mid \\
\nabla f_{1}(\xx) & \nabla f_{2}(\xx) & \cdots & \nabla f_{n}(\xx)\\
\mid & \mid & & \mid \\
\end{pmatrix} \begin{pmatrix}1/n\\1/n \\ \vdots \\ 1/n\end{pmatrix}.
\end{align*}
Hence, we can use approximate matrix multiplication results as a fundamental primitive in RandNLA~\cite{mahoney2011randomized,drineas2006fast}, to probabilistically control the error in approximation of $\nabla F(\xx)$ by $\bgg(\xx)$, through uniform sampling of the columns and rows of the involved matrices above. 

\begin{lemma}[Uniform Gradient Sub-Sampling]
\label{randnla_lemma}
For a given $\xx \in \mathcal{D} \cap \mathcal{X}$, let 
\begin{equation*}
\|\nabla f_{i}(\xx) \|_{\mathcal{K}} \leq G(\xx), \quad i =1,2,\ldots,n.
\end{equation*} 
For any $0 < \epsilon < 1$ and $0 < \delta < 1$, if
\begin{equation}
|\mathcal{S}| \geq \frac{G(\xx)^{2}}{\epsilon^{2}} \Big( 1 + \sqrt{8 \ln \frac{1}{\delta}}\Big)^{2},
\label{uniform_sample_size_gradient}
\end{equation}
then for $\bgg(\xx)$ defined in~\eqref{subsampled_G}, we have
\begin{equation*}
\Pr \Big( \|\nabla F(\xx) - \bgg(\xx) \|_{\mathcal{K}} \leq \epsilon \Big) \geq 1-\delta,
\end{equation*}
where $\|.\|_{\mathcal{K}}$ is defined as in~\eqref{norm_cone_vec}.
\end{lemma}

\comment \label{grad_sample_size_alg} Note that for the results and algorithms below which make use of Assumption~\eqref{grad_boundedness_opt}, we use the sample size, $|\mathcal{S}_{\bgg}|$, given by Lemma~\ref{randnla_lemma} with $G^{*} = G(\xx^{*})$ related to a local optimum, $\xx^{*}$. However, for those which do not make such assumption, the sample size, $|\mathcal{S}_{\bgg}|$, from Lemma~\ref{randnla_lemma} is given with the current iterate, $\xx^{(k)}$. For these latter algorithms, we need to be able to efficiently estimate $G(\xx^{(k)})$ at every iteration or, a priori, have a uniform upper bound for $G(\xx)$ for all $\xx \in \mathcal{D} \cap \mathcal{X}$. Fortunately, in many different problems, it is often possible to estimate $G(\xx)$ very efficiently; see Section~\ref{sec:examples} and SSN1~\cite[Section 4]{romassn1} for concrete examples of a uniform and per-iteration estimations, respectively.

\subsection{Main Results: Sub-Sampled Hessian \& Gradient}
\label{sec:main_result_grad}
In this section, we present various algorithms which incorporate both Hessian and gradient sub-sampling. In such a setting, one is generally faced with two options: either to sub-sample the gradient and the Hessian independently of each other, or to use the same collection of indices and perform simultaneous sub-sampling for both. In this section, we present algorithms and convergence results for both of these strategies. Combining the sampling of Lemma~\ref{hoeffding_lemma} for the Hessian that of Lemma~\ref{randnla_lemma} for the gradient, we are now in the position to present our main results for this Section.  We remind that, for this Section, we consider the general update of~\eqref{structural_update_grad}, where $H(\xx^{(k)})$ and $\bgg(\xx^{(k)})$ are sub-sampled approximations to $\nabla^{2} F(\xx^{(k)})$ and $\nabla F(\xx^{(k)})$, respectively. 

\begin{theorem}[Error Recursion of~\eqref{structural_update_grad}: Independent Sub-Sampling]
\label{uniform_newton_convergence_grad}
\begin{enumerate}
	\item \textbf{Global Regularity:} Let Assumptions~\eqref{F_Lip} and~\eqref{strong_convex_boundedness} hold, and let $0 < \delta < 1$, $0< \epsilon_{1} < 1$, and $0< \epsilon_{2} < 1$ be given. Set $|\mathcal{S}_{H}|$ as described in Comment~\ref{sample_size_alg} with $(\epsilon_{1},\delta)$ and $|\mathcal{S}_{\bgg}|$ as in~\eqref{uniform_sample_size_gradient} with $(\epsilon_{2},\delta)$ and $\xx^{(k)}$. Independently, choose $\mathcal{S}_{H}$ and $\mathcal{S}_{\bgg}$, and let $H(\xx^{(k)})$ and $\bgg(\xx^{(k)})$ be as in~\eqref{subsampled_H} and~\eqref{subsampled_G}, respectively. 
Then, for the update~\eqref{structural_update_grad} with $\alpha_{k} = 1$, with probability $(1-\delta)^{2}$, we have
\begin{equation*}
\|\xx^{(k+1)} - \xx^{*}\| \leq \eta + \rho_{0}\|\xx^{(k)} - \xx^{*}\| + \xi \|\xx^{(k)} - \xx^{*}\|^{2},
\end{equation*}
where 
\begin{equation*}
\eta = \frac{\epsilon_{2}}{(1-\epsilon_{1}) \gamma }, \quad 
\rho_{0} = \frac{\epsilon_{1}}{(1-\epsilon_{1})}, \quad \text{and} \quad 
\xi = \frac{L}{2(1-\epsilon_{1}) \gamma }.
\end{equation*}
\item \textbf{Local Regularity:} Under Assumptions~\eqref{F_Lip},~\eqref{strong_convex_boundedness_opt}, and~\eqref{initial_cond_spec} with $\epsilon_{1}$, the results of Theorem~\ref{uniform_newton_convergence_grad} holds with 
\begin{equation*}
\eta = \frac{2 \epsilon_{2}}{(1-\epsilon_{1}) \gamma^{*} }, \quad
\rho_{0} = \frac{2\epsilon_{1}}{(1-\epsilon_{1})}, \quad \text{and} \quad 
\xi = \frac{3 L}{(1-\epsilon_{1}) \gamma^{*} }.
\end{equation*}
\end{enumerate}
\end{theorem}

\comment Similar to the case of Hessian sub-sampling, the bounds given here exhibit a composite behavior where the error is at first dominated by a quadratic term, which transforms to linear term, and finally is dominated by the rate at which the gradient is approximated near an optimum. 

It is also possible to obtain an alternative convergence result under local regularity assumption on the gradients at a local optimum $\xx^{*}$ as in~\eqref{grad_boundedness_opt}. This gives rise to an algorithm where the Hessian and the gradient are simultaneously sub-sampled, resulting in dependency among the approximations; see Algorithm~\ref{alg_simult_hessian_grad}.

\begin{theorem}[Error Recursion of~\eqref{structural_update_grad}: Simultaneous Sub-Sampling]
\label{uniform_newton_convergence_grad_relax_2}
Let Assumptions~\eqref{F_Lip},~\eqref{strong_convex_boundedness_opt}, and~\eqref{grad_boundedness_opt} hold, and let $0 < \delta < 1$ and $0< \epsilon < 1$ be given. Using the same $\epsilon$, set $|\mathcal{S}_{H}|$ as described in Comment~\ref{sample_size_alg} and $|\mathcal{S}_{\bgg}|$ as in~\eqref{uniform_sample_size_gradient} with $G^{*}$. Select a sample set $\mathcal{S}$ of size $\max\{|\mathcal{S}_{H}|, |\mathcal{S}_{G}|\}$, uniformly at random and let $H(\xx^{(k)})$ and $\bgg(\xx^{(k)})$ be as in~\eqref{subsampled_H} and~\eqref{subsampled_G}, respectively, both using $\mathcal{S}$. For the update~\eqref{structural_update_grad} with $\alpha_{k} = 1$, if $\xx^{(k)}$ satisfies~\eqref{initial_cond_spec}, then with probability $1-2 \delta$, we have
\begin{equation*}
\|\xx^{(k+1)} - \xx^{*}\| \leq \eta + \xi\|\xx^{(k)} - \xx^{*}\|^{2},
\end{equation*}
where 
\begin{equation*}
\eta = \frac{2 \epsilon}{(1-\epsilon) \gamma^{*}}, \quad
\xi = \frac{L }{(1-\epsilon) \gamma^{*}} .
\end{equation*}
\end{theorem}

\begin{algorithm}
\caption{Newton with Independent Sub-Sampling of Hessian and Gradient}
\begin{algorithmic}[1]
\STATE \textbf{Input:} $\xx^{(0)}$, $0 < \delta < 1$, $0 < \epsilon_{1} < 1$, $0 < \epsilon_{2} < 1$ and $0 < \rho < 1$
\STATE - Set the sample size, $|\mathcal{S}_{H}|$, with $\epsilon_{1}$ and $\delta$ as described in Comment~\ref{sample_size_alg} 
\FOR {$k = 0,1,2, \cdots$ until termination} 
\STATE - Select a sample set, $\mathcal{S}_{H}$, of size $|\mathcal{S}_{H}|$ and form $H(\xx^{(k)})$ as in~\eqref{subsampled_H}
\STATE - Set $\epsilon^{(k)}_{2} = \rho^{k} \epsilon_{2}$
\STATE - Set the sample size, $|\mathcal{S}^{(k)}_{\bgg}|$, with $\epsilon^{(k)}_{2}$, $\delta$ and $\xx^{(k)}$ as in~\eqref{uniform_sample_size_gradient}
\STATE - Select a sample set, $\mathcal{S}^{(k)}_{\bgg}$ of size $|\mathcal{S}_{\bgg}^{(k)}|$ and form $\bgg(\xx^{(k)})$ as in~\eqref{subsampled_G}
\STATE - Update $\xx^{(k+1)}$ as in~\eqref{structural_update_grad} with $H(\xx^{(k)})$, $\bgg(\xx^{(k)})$ and $\alpha_{k}=1$
\ENDFOR
\end{algorithmic}
\label{alg_hessian_grad}
\end{algorithm}

We are now in the position to give R-linear convergence of the proposed algorithms.  

\begin{theorem}[R-Linear Convergence of Algorithm~\ref{alg_hessian_grad}] 
\label{uniform_newton_grad_sufficient_cond}
\begin{enumerate}
	\item \textbf{Global Regularity:} Let Assumptions~\eqref{F_Lip} and~\eqref{strong_convex_boundedness} hold. Consider any $0 < \rho < 1$, $0 < \rho_{0} < 1$, and $0 < \rho_{1} < 1$ such that $\rho_{0} + \rho_{1} < \rho$. Let
\begin{equation*}
\epsilon_{1} \leq \frac{\rho_{0} }{1 + \rho_{0} },
\end{equation*}	
and define
\begin{equation*}
\sigma \defeq \frac{2 ( \rho - (\rho_{0} + \rho_{1}) ) (1-\epsilon_{1}) \gamma }{L}.
\end{equation*}	
Using Algorithm~\ref{alg_hessian_grad} with
\begin{equation*}
\epsilon_{2} \leq (1-\epsilon_{1})\gamma \rho_{1} \sigma,
\end{equation*}
if the initial iterate, $\xx^{(0)}$, satisfies 
\begin{equation}
\|\xx^{(0)}-\xx^{*}\| \leq \sigma, 
\label{initial_cond_grad}
\end{equation}
we get locally R-linear convergence
\begin{equation}
\|\xx^{(k)} - \xx^{*}\| \leq \rho^{k} \sigma,
\label{lin_grad}
\end{equation} 
with probability $(1-\delta)^{2k}$. 
\item \textbf{Local Regularity:} Under Assumptions~\eqref{F_Lip} and~\eqref{strong_convex_boundedness_opt}, Theorem~\ref{uniform_newton_grad_sufficient_cond} holds with 
\begin{align*}
\epsilon_{1} \leq \frac{\rho_{0} }{2 + \rho_{0}}, \quad \text{and} \quad
\epsilon_{2} \leq (1-\epsilon_{1})\gamma^{*} \rho_{1} \sigma/2,
\end{align*}
where
\begin{equation*}
\sigma \defeq \frac{( \rho - (\rho_{0} + \rho_{1}) ) (1-\epsilon_{1}) \gamma^{*} }{3L}.
\end{equation*}
\end{enumerate}
\end{theorem}


\comment From Theorems~\ref{uniform_newton_grad_sufficient_cond}, it can be seen that in order to get linear convergence rate, estimation of the gradient must be done, progressively, more accurately, whereas the sample size for Hessian can remain unchanged across iterations. This is in line with the common knowledge where, in practice, as the iterations get closer to the optimal solution, the accuracy of gradient estimation is more important than that of Hessian. 

\begin{algorithm}
\caption{Newton with Simultaneous Sub-Sampling of Hessian and Gradient}
\begin{algorithmic}[1]
\STATE \textbf{Input:} $\xx^{(0)}$, $0 < \delta < 1$, $0 < \epsilon < 1$, $0 < \rho < 1$
\FOR {$k = 0,1,2, \cdots$ until termination} 
\STATE - Set $\epsilon^{(k)} = \rho^{k} \epsilon$ 
\STATE - Set the sample size, $|\mathcal{S}^{(k)}_{H}|$, with $\epsilon^{(k)}$ and $\delta$ as described in Comment~\ref{sample_size_alg}
\STATE - Set the sample size, $|\mathcal{S}^{(k)}_{\bgg}|$, with $\epsilon^{(k)}$ and $\delta$ as described in Comment~\ref{grad_sample_size_alg}
\STATE - Set $|\mathcal{S}^{(k)}| = \max \{|\mathcal{S}^{(k)}_{\bgg}|, |\mathcal{S}^{(k)}_{H}|\}$
\STATE - Select a sample set, $\mathcal{S}^{(k)}$, of size $|\mathcal{S}^{(k)}|$
\STATE - Form $H(\xx^{(k)})$ as in~\eqref{subsampled_H} and $\bgg(\xx^{(k)})$ as in~\eqref{subsampled_G}, both using $\mathcal{S}^{(k)}$
\STATE - Update $\xx^{(k+1)}$ as in~\eqref{structural_update_grad} with $H(\xx^{(k)})$, $\bgg(\xx^{(k)})$ and $\alpha_{k}=1$
\ENDFOR
\end{algorithmic}
\label{alg_simult_hessian_grad}
\end{algorithm}

\begin{theorem}[R-Linear Convergence of Algorithm~\ref{alg_simult_hessian_grad}]
\label{uniform_newton_grad_sufficient_cond_relax_2}
Let Assumptions~\eqref{F_Lip},~\eqref{strong_convex_boundedness_opt} and~\eqref{grad_boundedness_opt} hold and consider any $0 < \rho_{0} < \rho < 1$.  Let $0 < \epsilon < 1$ be such that
\begin{equation}
\frac{\epsilon}{(1-\epsilon)^{2}} \leq \frac{\rho_{0} (\rho - \rho_{0}) (\gamma^{*})^{2} }{4L},
\label{epsilon_cond}
\end{equation}
and define
\begin{equation}
\sigma \defeq \frac{(\rho - \rho_{0})(1-\epsilon) \gamma^{*}}{2L}. 
\label{initial_cond_grad_2}
\end{equation}
Using Algorithm~\ref{alg_simult_hessian_grad}, if the initial iterate, $\xx^{(0)}$, satisfies
\begin{equation*}
\|\xx^{(0)} - \xx^{*} \| \leq \sigma, 
\end{equation*}
we get locally R-linear convergence as in~\eqref{lin_grad} with probability $(1-2 \delta)^{k}$. 
\end{theorem}

\comment Unlike Theorem~\ref{uniform_newton_grad_sufficient_cond}, the assumptions in Theorem~\ref{uniform_newton_grad_sufficient_cond_relax_2} are more relaxed and the sampling for both are done simultaneously, i.e., using the same sample collection $\mathcal{S}$. As a result in order to guarantee linear convergence, using Algorithm~\ref{alg_simult_hessian_grad}, sampling the gradient and the Hessian both must be done, progressively, more accurately.

Using a slight modification of Algorithm~\ref{alg_simult_hessian_grad}, it is even possible to obtain R-superlinear convergence. However, in this case, the sample size increase must be done ``extremely aggressively'', which might render such strategy rather impractical and we only give the following result for completeness. By aggressive, we refer to a very fast rate at which $\epsilon^{(k)}$ is decreased across iterations, compared to a more moderate decrease of the previous results.

\begin{theorem}[R-Superlinear Convergence of Algorithm~\ref{alg_simult_hessian_grad}]
\label{uniform_newton_grad_sufficient_cond_relax_3}
Under the assumptions of Theorem~\ref{uniform_newton_grad_sufficient_cond_relax_2}, Algorithm~\ref{alg_simult_hessian_grad} with 
\begin{align*}
\epsilon^{(0)} &= \epsilon, \\ 
\epsilon^{(k)} &= \rho^{k} \epsilon^{(k-1)}, \quad k=1,2,\ldots
\end{align*}
converges locally R-superlinearly as
\begin{subequations}
\label{sup_lin_grad}
\begin{equation}
\|\xx^{(k)} - \xx^{*}\| \leq \tau^{(k)} \sigma,
\end{equation} 
with probability $(1-2 \delta)^{k}$, where the sequence $\{\tau^{(k)}\}$ satisfies
\begin{align}
\tau^{(1)} &= \rho, \\
\frac{\tau^{(k)}}{\tau^{(k-1)}} &= \rho^{k-1}, \quad k = 2,3,\ldots.
\end{align}
\end{subequations}
\end{theorem}

\section{Examples}
\label{sec:examples}
In this Section, we present a few instances of problems which are of the form~\eqref{obj}. Specifically, examples from generalized linear models (GLM) and linear support vector machine (SVM) are given in Sections~\ref{sec:example_glm} and~\ref{sec:example_svm} respectively.

\subsection{GLMS with Sparsity Constraint}
\label{sec:example_glm}
The class of generalized linear models is used to model a wide variety of regression and classification problems. The process of data fitting using such GLMs usually consists of a training data set containing $n$ response-covariate pairs, and the goal is to predict some output response based on some covariate vector, which is given after the training phase. More specifically, let $(\aa_{i},b_{i}), \; i = 1,2,\cdots,n,$ form such response-covariate  pairs in the training set where $\aa_{i} \in \mathbb{R}^{p}$. The domain of $b_{i}$ depends on the GLM used: for example, in the standard linear Gaussian model $b_{i} \in \mathbb{R}$, in the logistic models for classification, $b_{i} \in \{0, 1\}$, and in Poisson models for count-valued responses, $b_{i} \in \{0, 1, 2, \ldots \}$. See the book~\cite{mccullagh1989generalized} for further details and applications.

Consider \textit{sparse} maximum likelihood (ML) estimation using any GLM with canonical link function. One way to induce sparsity is by using the constraint set, $\mathcal{X}$, as 
\begin{equation*}
\mathcal{X} = \{\xx \in \mathbb{R}^{p}; \|\xx\|_{1} \leq 1 \}.
\end{equation*}
Hence, sparse ML is equivalent to minimizing the negative log-likelihood over $\mathcal{X}$, where the negative log-likelihood of such model can be written as 
\begin{equation*}
F(\xx) = \frac{1}{n} \sum_{i=1}^{n} \left( \Phi(\aa_{i}^{T} \xx) - b_{i} \aa_{i}^{T} \xx \right).
\end{equation*}
The \textit{cumulant generating function}, $\Phi$, determines the type of GLM. For example, $\Phi(t) = 0.5 t^{2}$ gives rise to ordinary least squares formulation (OLS), while $\Phi(t) = \ln\left(1+\exp(t)\right)$ and $\Phi(t) = \exp(t)$ yield logistic regression (LR) and Poisson regression (PR), respectively. It is also easily verified that the gradient and the Hessian of $F$ are
\begin{eqnarray*}
\nabla F(\xx) &=&  \frac{1}{n} \sum_{i=1}^{n} \left( \frac{d \Phi(t)}{d t} \rvert_{t = \aa_{i}^{T} \xx} - b_{i} \right) \aa_{i}, \\
\nabla^{2} F(\xx) &=&  \frac{1}{n} \sum_{i=1}^{n} \left(  \frac{d^{2} \Phi(t)}{d t^{2}} \rvert_{t = \aa_{i}^{T} \xx} \right) \aa_{i} \aa^{T}_{i}.
\end{eqnarray*}

As mentioned in Comment~\ref{grad_sample_size_alg}, in order to use Lemma~\ref{randnla_lemma}, we need to be able to efficiently estimate $G(\xx^{(k)})$ at every iteration or, a priori, have a uniform upper bound for $G(\xx)$ for all $\xx \in \mathcal{D} \cap \mathcal{X}$.  For illustration purposes only, Table~\ref{table_GLM} gives some very rough estimates of the constant $G$ for GLMs with respect to $\mathcal{X}$. Note that, here, $G$ is a uniform bound for $G(\xx)$. For an example where a per-iteration bound can be efficiently computed see Section~\ref{sec:example_svm}. 

\begin{table}[htb]
\centering
\scalebox{1}{
 \begin{tabular}{||c || c || c || c ||} 
 \hline
  & OLS & LR &  PR  \\ [0.5ex] 
 \hline\hline
 $\nabla f_{i}(\xx)$ & $\left( \aa_{i}^{T} \xx - b_{i} \right) \aa_{i}$ & $( \frac{1}{1+e^{-\aa_{i}^{T} \xx}}  - b_{i} ) \aa_{i}$ & $( e^{\aa_{i}^{T} \xx} - b_{i} ) \aa_{i}$ \\ [2ex]
$G$ & $\max_{i} (\|\aa_{i}\|_{\infty} + |b_{i}|) \|\aa_{i}\|$ & $\max_{i} (\frac{1}{1+e^{-\|\aa_{i}\|_{\infty}}}  + |b_{i}|) \|\aa_{i}\|$ & $\max_{i} (e^{\|\aa_{i}\|_{\infty}} + |b_{i}|) \|\aa_{i}\|$ \\ [2ex] 
 \hline
 \end{tabular}}
\caption{Estimates for $G$ in GLMs with sparsity constraint
\label{table_GLM}}
\end{table}

\subsection{SVM with Smooth Quadratic Hinge Loss}
\label{sec:example_svm}
A support vector machine constructs a hyperplane or set of hyperplanes in a high-dimensional space, which can be used for classification, regression, or other tasks. The vast majority of text books and articles introducing SVM very briefly state the the primal optimization problem, and then go directly to the dual formulation~\cite{vapnik1998statistical,burges1998tutorial,scholkopf2002learning}. Primal optimizations
of linear SVMs have already been studied by~\cite{keerthi2005modified,chapelle2007training,mangasarian2002finite}. This is so since the dual problem does not scale well with the number of data points, e.g., $\mathcal{O}(n^{3})$ for some approaches,~\cite{woodsend2011exploiting}, the primal might be better-suited for optimization of linear SVMs.

Given a training set ${(\aa_{i}, b_{i})}, \; i=1,2,\ldots,n$, $\aa_{i} \in \mathbb{R}^{p}$ and $b_{i} \in \{+1, −1\}$, recall that the primal
SVM optimization problem is usually written as:
\begin{equation*}
\begin{aligned}
\min_{\xx} &\quad \frac{C}{n} \sum_{i=1}^{n} c_{i}^{m} + \hf \| \xx \|^{2}\\
\text{s.t.} & \quad b_{i} \xx^{T}\aa_{i} \geq 1-c_{i}; \quad \forall i \\
& \quad c_{i} \geq 0; \quad \forall i
\end{aligned}
\end{equation*}
where $C > 0$ and $m$ is either $1$ (hinge loss) or 2 (quadratic loss). Note that, here, the coordinate corresponding to the intercept is implicitly included in each $\aa_{i}$. The unconstrained optimization formulation of the above is 
\begin{equation*}
\min_{\xx,y} \frac{C}{n} \sum_{i=1}^{n} \ell(b_{i}, \xx^{T}\aa_{i}) + \hf \| \xx \|^{2},
\end{equation*}
with $\ell(v,w) = \max(0,1-vw)^{m}$ or any other loss function (e.g. Huber loss function). For sub-sampled newton methods where we require $\ell$ to be smooth, we consider the case $m=2$, yielding the SVM with smooth quadratic hinge Loss. Now denoting 
\begin{equation*}
\mathcal{I} = \{i; \; b_{i} \xx^{T}\aa_{i} < 1\},
\end{equation*}
the gradient and the Hessian of $f_{i}$ is given as follows
\begin{eqnarray*}
\nabla f_{i}(\xx) &=& \left ( 2 C (\xx^{T}\aa_{i} - b_{i}) \aa_{i} \right) \mathbbm{1}_{i \in \mathcal{I}} + \xx \\
\nabla^{2} f_{i}(\xx) &=& \left(2 C \aa_{i}\aa_{i}^{T}\right) \mathbbm{1}_{i \in \mathcal{I}} + \mathbb{I},
\end{eqnarray*}
where $\mathbbm{1}_{i \in \mathcal{I}}$ is the indicator function of the set $\mathcal{I}$. A very rough estimate of $G(\xx)$ can be obtained by
\begin{equation*}
G(\xx) \leq \|\xx\| \max_{i} (2 C \|\aa_{i}^{T}\|^{2} + 1 ) + 2C\max_{i} |b_{i}| \|\aa_{i}\|.
\end{equation*}
As mentioned before, in practice, as a pre-processing and before starting the algorithms, one can pre-compute the quantities which depend only on $(\aa_{i}, b_{i})$'s. Then updating $G(\xx)$ at every iteration is done very efficiently as it is just a matter of computing $\|\xx\|$.


\section{Conclusion}
\label{sec:conclusion}
In this paper, we studied the local convergence behavior of sub-sampled Newton methods for constrained optimization in two general settings. We first studied the convergence behavior  of the algorithm in which only the Hessian is sub-sampled and the full gradient is used. We showed that the error recursion has a composite nature where it is first dominated by a quadratic term and is subsequently transformed to  a linear term near an optimum. In such a setting, we proved locally Q-linear convergence of the proposed algorithm. We also showed that by increasing the accuracy of such sub-sampling as the iterations progress, one can indeed recover superlinear rates. We then studied the situations in which the sub-sampled Hessian is regularized by either modifying its spectrum or by ridge-type regularization. We showed that such regularization can only be beneficial at early stages of the algorithm, i.e., the phase where the quadratic term dominates the error recursion, and we need to revert to using the unmodified sub-sampled Hessian as iterates get closer to the optimum.

Finally, we studied the local convergence of a fully stochastic algorithm in which both Hessian and the gradients are sub-sampled. Such sub-sampling can be done independently for the Hessian and the gradients, or simultaneously where one sample collection of indices is used for sub-sampling both. As before, for all of our results, we showed the composite rate of the error recursions. We also showed that by progressively increasing the sub-sampling accuracy of the gradient, we can indeed guarantee a local R-linear convergence rate. In addition, we argued that through a more aggressive sub-sampling strategy, one can obtain a superlinear rate. 

For all of our results, we showed that the error bounds exhibit a composite behavior whose dominating term varies according to the distance of iterates to optimality. For example, for Hessian sub-sampling, we showed that, when far from a local optimum, the dominating error term is quadratic which then transitions to linear when the iterates are within a small enough region around that local optimum. These results might give a better control over various tradeoffs which exhibit themselves in different applications, e.g., practitioners  might require faster running-time while statisticians might be more interested in statistical aspects regarding the recovered solution.

One major drawback of our algorithms is that we require the quadratic sub-problems~\eqref{structural_update_grad} to be solved exactly. This can indeed be a computational challenge in many situations. In SSN1~\cite{romassn1} and in the context of globally convergent sub-sampled algorithms for unconstrained problems, we relax this requirement and solve such sub-problems only approximately. This results in globally convergent sub-sampled algorithms with inexact updates which have much lower per-iteration cost. The extension of such inexact updates for the sub-sampled algorithms for constrained optimization and studying their local convergence behavior are left for future work.


%
\bibliographystyle{plain}
\bibliography{../../../biblio}

\appendix
\section{Proofs}
\label{proofs}

\subsection{Proofs of the theorems and lemmas of Section~\ref{sec:sub_hessian}}
Here we give the proofs of all of the results in Section~\ref{sec:sub_hessian}.

\begin{proof}[Proof of Lemma~\ref{hoeffding_lemma}]
Let $U$ be an orthonormal basis for the cone $\mathcal{K}$ defined in~\eqref{cone}. Consider $|\mathcal{S}|$ i.i.d random matrices $H_{j}(\xx), j=1,2,\ldots,|\mathcal{S}|$ such that $\Pr(H_{j}(\xx) = \nabla^{2} f_{i}(\xx)) = 1/n; \; \forall i = 1,2,\ldots,n,$. Define
\begin{align*}
&X_{j} \defeq  U^{T} \Big( H_{j}  - \nabla^{2}F(\xx)\Big) U, \\
&H \defeq \frac{1}{|\mathcal{S}|} \sum_{j \in \mathcal{S}} H_{j}, \\
&X \defeq \sum_{j \in \mathcal{S}} X_{j} = |\mathcal{S}| U^{T} \Big( H - \nabla^{2}F(\xx) \Big) U.
\end{align*} 
Note that $\Ex(X_{j}) = 0$ and for $H_{j} = \nabla^{2} f_{1}(\xx)$ we have
\begin{equation*}
\|X_{j}^{2}\| = \|X_{j}\|^{2} \leq \| U^{T} \big(\frac{n-1}{n} \nabla^{2} f_{1}(\xx) - \sum_{i=2}^{n} \frac{1}{n} \nabla^{2} f_{i}(\xx) \big) U\|^{2}\leq 4 \left(\frac{n-1}{n}\right)^{2} K^{2} \leq 4 K^{2}.
\end{equation*}
Hence we can apply Operator-Bernstein inequality~\cite[Theorem 1]{gross2010note} to get
\begin{equation*}
\Pr\Big(\|H - \nabla^{2}F(\xx) \|_{\mathcal{K}} \geq \epsilon \gamma \Big) = \Pr\Big(\|X\|_{\mathcal{K}} \geq \epsilon |\mathcal{S}| \gamma \Big) \leq 2p \exp\{-\epsilon^{2} |\mathcal{S}| \gamma^{2} /(16 K^{2})\}.
\end{equation*}
Noting that for any symmetric matrix $A$,
\begin{equation*}
\|A\|_{\mathcal{K}} = \max \left\{ \lambda_{\max}^{\mathcal{K}}(A), - \lambda_{\min}^{\mathcal{K}}(A)\right\},
\end{equation*}
we define the ``good ''event 
\begin{align*}
\mathcal{A} \defeq &\Big\{ \|H - \nabla^{2}F(\xx) \|_{\mathcal{K}} \leq \epsilon \gamma \Big\} \\
 = &\Big\{ -\epsilon \gamma \leq \lambda_{\min}^{\mathcal{K}}\left(H - \nabla^{2}F(\xx) \right) \leq \lambda_{\max}^{\mathcal{K}}\left(H - \nabla^{2}F(\xx) \right) \leq \epsilon \gamma \Big\}.
\end{align*}
Now having a sample size as in~\eqref{uniform_sample_size_Hoeffding}, we get that 
\begin{equation*}
2 p \exp\{-\epsilon^{2} |\mathcal{S}| \gamma^{2} /(16 K^{2})\} \leq \delta,
\end{equation*}
which, in turn, gives
\begin{equation*}
\Pr(\mathcal{A}) \geq 1-\delta.
\end{equation*}
Consider eigenvalues of a matrix $X \in \mathbb{R}^{p \times p}$ ordered as $\lambda_{1}(X) \geq \lambda_{2}(X) \geq \cdots \geq \lambda_{p}(X)$. On the event, $\mathcal{A}$, and for any $1 \leq i,j \leq p$ such that $i+j = p+1$, by Weyl's inequality~\cite[Theorem 8.4.11, Eqn.\ (8.4.11)]{bernstein2009matrix}, we have 
\begin{align*}
\lambda_{i}^{\mathcal{K}} \left( H(\xx) \right) - \lambda_{i}^{\mathcal{K}} \left( \nabla^{2}F(\xx) \right)  &= \lambda_{i}^{\mathcal{K}} \left( H(\xx) \right) + \lambda_{j}^{\mathcal{K}} \left( - \nabla^{2}F(\xx) \right) \\
&\leq \lambda_{\max}^{\mathcal{K}} \left( H(\xx) - \nabla^{2}F(\xx) \right) \\
&\leq \epsilon \gamma \\
&\leq \epsilon \lambda_{i}^{\mathcal{K}} \left( \nabla^{2}F(\xx) \right).
\end{align*}
Hence, we get
\begin{equation*}
\lambda_{i}^{\mathcal{K}} \left( H(\xx) \right) \leq (1+\epsilon) \lambda_{i}^{\mathcal{K}} \left( \nabla^{2}F(\xx) \right).
\end{equation*}
Similarly, using~\cite[Theorem 8.4.11, Eqn.\ (8.4.12)]{bernstein2009matrix}, and for any $1 \leq i,j \leq p$ such that $i+j = p+1$, we have 
\begin{align*}
\lambda_{i}^{\mathcal{K}} \left( H(\xx) \right) - \lambda_{i}^{\mathcal{K}} \left( \nabla^{2}F(\xx) \right)  &= \lambda_{i}^{\mathcal{K}} \left( H(\xx) \right) + \lambda_{j}^{\mathcal{K}} \left( - \nabla^{2}F(\xx) \right)\\
&\geq \lambda_{\min}^{\mathcal{K}} \left( H(\xx) - \nabla^{2}F(\xx) \right) \\
&\geq - \epsilon \gamma \\
&\geq - \epsilon \lambda_{i}^{\mathcal{K}} \left( - \nabla^{2}F(\xx) \right).
\end{align*}
Hence, we get
\begin{equation*}
\lambda_{i}^{\mathcal{K}} \left( H(\xx) \right) \geq (1-\epsilon) \lambda_{i}^{\mathcal{K}} \left( \nabla^{2}F(\xx) \right), \quad \forall i \in \{ 1,2,\ldots,p\},
\end{equation*}
and the result follows.\qed
\end{proof}

\begin{proof}[Proof of Lemma~\ref{hoeffding_lemma_const}]
Let $X_{j}$ be as in the proof of Lemma~\ref{hoeffding_lemma}. Since $f_{i}$ is convex, we have $\nabla^{2} f_{i}(\xx) \succeq 0$, $\nabla^{2} F(\xx) \succeq 0$, and for $H_{j} =  \nabla^{2} f_{i}(\xx)$, we have
\begin{equation*}
-K \mathbb{I} \preceq - U^{T} \nabla^{2} F(\xx) U \preceq X_{j} \preceq  U^{T}\nabla^{2} f_{i}(\xx) U \preceq K \mathbb{I}.
\end{equation*}
So, it follows that 
\begin{equation*}
X^{2}_{j} \preceq K^{2} \mathbb{I},
\end{equation*}
and hence, after obtaining the same bound for $Y_{j}=-X_{j}$ and applying Operator-Bernstein inequalit as in the proof of Lemma~\ref{hoeffding_lemma}, we get a lower bound for the sample size as in~\eqref{uniform_sample_size_Hoeffding_const}. 
\qed
\end{proof}

\begin{proof}[Proof of Lemma~\ref{hoeffding_lemma_intrinsic}]
Let $X_{j}$ be as in the proof of Lemma~\ref{hoeffding_lemma}. We have
\begin{eqnarray*}
\Ex(X^{2}_{j}) &=& \Ex\left[ \left( U^{T} \left( \nabla^{2}F(\xx) - H_{j} \right) U \right)^{2} \right]\\
&=& (U^{T} \nabla^{2}F(\xx) U)^{2} - U^{T} \nabla^{2}F(\xx) U U^{T} \Ex (H_{j}) U \\
&& - U^{T} \Ex (H_{j}) U U^{T} \nabla^{2}F(\xx) U +  \Ex (U^{T} H_{j} U)^{2} \\
&=& \frac{1}{n} \sum_{i=1}^{n} (U^{T} \nabla^{2}f_{i}(\xx) U)^{2} - (U^{T} \nabla^{2}F(\xx) U)^{2} \\
&\preceq& \frac{1}{n} \sum_{i=1}^{n} (U^{T} \nabla^{2}f_{i}(\xx) U)^{2} = \frac{V}{|\mathcal{S}|}.
\end{eqnarray*}
We see that 
\begin{equation*}
\|V\| = \|\frac{|\mathcal{S}|}{n} \sum_{i=1}^{n} (U^{T} \nabla^{2}f_{i}(\xx) U)^{2}\| \leq \frac{|\mathcal{S}|}{n} \sum_{i=1}^{n} \|U^{T} \nabla^{2}f_{i}(\xx) U\|^{2} \leq |\mathcal{S}| K^{2}.
\end{equation*}
Hence if 
\begin{equation}
\epsilon |\mathcal{S}| \gamma \geq (\sqrt{|\mathcal{S}|}  + \frac{2}{3})K
\label{mat_bern_t}
\end{equation}
we can apply Matrix Bernstein using the intrinsic dimension~\cite[Theorem 7.7.1]{tropp2015introduction}, to get for $\epsilon \leq 1/2$,
\begin{eqnarray*}
\Pr\Big(\lambda_{\max}(X) \geq \epsilon |\mathcal{S}| \gamma \Big) &\leq& 4d \exp\{-\epsilon^{2} |\mathcal{S}| \gamma^{2} /(2 K^{2} + 4 \epsilon \gamma K/3 ) \} \\
&\leq& 4d \exp\{-3 \epsilon^{2} |\mathcal{S}| \gamma^{2} /(16 K^{2})\}. 
\end{eqnarray*}
Applying the same bound for $Y_{j} = -X_{j}$ and $Y = \sum_{j \in \mathcal{S}} {Y_{j}}$, followed by using the union bound, we get the desired result. It is also easily verified that~\eqref{uniform_sample_size_Bernstein} implies~\eqref{mat_bern_t}, and so our derivation is valid. 
\qed
\end{proof}

Note that for the simplicity of the presentation of the proofs of our main results, we will make use of the following lemma, which is just a simple variant of Lemma~\ref{hoeffding_lemma}, and whose proof is very similar to that of Lemma~\ref{hoeffding_lemma} 
\begin{lemma}[Uniform Hessian Sub-Sampling: Matrix Hoeffding]
\label{hoeffding_lemma_2}
Given any $0 < \epsilon < 1$, $0 < \delta < 1$  and $\xx \in \mathcal{D} \cap \mathcal{X}$, if the sample size $|\mathcal{S}|$ is as in~\eqref{uniform_sample_size_Hoeffding}, then for $H(\xx)$ defined in~\eqref{subsampled_H}, we have all of the following statements, simultaneously, with probability $1-\delta$:
\begin{align}
&\|H - \nabla^{2}F(\xx) \|_{\mathcal{K}} \leq \epsilon \gamma, \label{hoeffding_absolute_error}\\
& \lambda_{\min}^{\mathcal{K}} \left( H(\xx) \right) \geq (1-\epsilon) \gamma , \label{hoeffding_invertibility} \\
& \frac{\|H - \nabla^{2}F(\xx) \|_{\mathcal{K}}}{\lambda_{\min}^{\mathcal{K}} \left( H(\xx) \right)} \leq \frac{\epsilon}{1-\epsilon} . \label{linear_rate_factor}
\end{align}
\end{lemma}

\subsubsection{Structural Lemmas}
\label{sec:strutural_lemmas_hess}
We first give some structural lemma which will be the foundation of our main results for this Section.
\begin{lemma}[Structural Lemma 1]
\label{structural_lemma_eig}
Let Assumptions~\eqref{F_Lip} and~\eqref{strong_convex_boundedness} hold. Also assume that
\begin{equation}
\pp^{T} H(\xx^{(k)}) \pp  > 0 , \quad \forall \; \pp \in \mathcal{K} \setminus \{0\}.
\label{H_strict_convex}
\end{equation}
For the update~\eqref{structural_update}, we have
\begin{equation*}
\|\xx^{(k+1)} - \xx^{*}\| \leq \rho_{0}\|\xx^{(k)} - \xx^{*}\| + \xi  \|\xx^{(k)} - \xx^{*}\|^{2},
\end{equation*}
where
\begin{equation*}
\rho_{0} \defeq \frac{\left\| H(\xx^{(k)}) - \alpha_{k} \nabla^{2}F(\xx^{(k)}) \right\|_{\mathcal{K}} }{\lambda_{\min}^{\mathcal{K}} \left( H(\xx^{(k)}) \right)}, \quad \xi \defeq \frac{\alpha_{k} L}{2 \lambda_{\min}^{\mathcal{K}} \left( H(\xx^{(k)}) \right)}.
\end{equation*}
\end{lemma}
\begin{proof}
Define $\Delta_{k} \defeq \xx^{(k)} - \xx^{*}$. By optimality of $\xx^{(k+1)}$ in~\eqref{structural_update}, we have for any $\xx \in \mathcal{D} \cap \mathcal{X}$,
\begin{equation*}
(\xx - \xx^{(k+1)})^{T}\nabla F(\xx^{(k)}) + \frac{1}{\alpha_{k}}(\xx - \xx^{(k+1)})^{T} H(\xx^{(k)}) (\xx^{(k+1)} - \xx^{(k)}) \geq 0.
\end{equation*}
In particular, setting $\xx = \xx^{*}$, and noting that $\xx^{(k+1)} - \xx^{(k)} = \Delta_{k+1} - \Delta_{k}$, we get
\begin{equation*}
\Delta_{k+1}^{T} H(\xx^{(k)}) \Delta_{k+1} \leq \Delta_{k+1}^{T} H(\xx^{(k)}) \Delta_{k} - \alpha_{k} \Delta_{k+1}^{T} \nabla F(\xx^{(k)}).
\end{equation*}
%
Since by optimality of $\xx^{*}$, we have $\nabla F(\xx^{*})^{T} (\xx^{(k+1)} - \xx^{*}) \geq 0$, it follows that
\begin{equation*}
\Delta_{k+1}^{T} H(\xx^{(k)}) \Delta_{k+1} \leq \Delta_{k+1}^{T} H(\xx^{(k)}) \Delta_{k} - \alpha_{k}\Delta_{k+1}^{T} \nabla F(\xx^{(k)}) + \alpha_{k}\Delta_{k+1}^{T} \nabla F(\xx^{*}).
\end{equation*}
Now, by the relation
\begin{equation*}
\nabla F(\xx^{(k)}) - \nabla F(\xx^{*}) = \left(\int_{0}^{1} \nabla^{2} F\big(\xx^{*} + t(\xx^{(k)} - \xx^{*})\big) d t \right) (\xx^{(k)} - \xx^{*}) ,
\end{equation*}
we have
\begin{eqnarray*}
\Delta_{k+1}^{T} H(\xx^{(k)}) \Delta_{k+1} &\leq& \Delta_{k+1}^{T} H(\xx^{(k)}) \Delta_{k} - \alpha_{k}\Delta_{k+1}^{T} \left(\int_{0}^{1} \nabla^{2} F\big(\xx^{*} + t(\xx^{(k)} - \xx^{*})\big) d t \right) \Delta_{k} \\
&=& \Delta_{k+1}^{T} H(\xx^{(k)})  \Delta_{k} - \alpha_{k}\Delta_{k+1}^{T} \nabla^{2} F(\xx^{(k)}) \Delta_{k} \\
&& + \alpha_{k}\Delta_{k+1}^{T} \nabla^{2} F(\xx^{(k)}) \Delta_{k} - \alpha_{k}\Delta_{k+1}^{T} \left(\int_{0}^{1} \nabla^{2} F\big(\xx^{*} + t(\xx^{(k)} - \xx^{*})\big) d t \right) \Delta_{k} \\
&=& \Delta_{k+1}^{T} \Big( H(\xx^{(k)}) - \alpha_{k} \nabla^{2}F(\xx^{(k)}) \Big) \Delta_{k} \\
&&  + \alpha_{k}\Delta_{k+1}^{T} \Big( \int_{0}^{1} \nabla^{2} F(\xx^{(k)}) - \nabla^{2} F\big(\xx^{*} + t(\xx^{(k)} - \xx^{*})\big) d t \Big) \Delta_{k} \\
&\leq& \left\| H(\xx^{(k)}) - \alpha_{k} \nabla^{2}F(\xx^{(k)}) \right\|_{\mathcal{K}} \|\Delta_{k}\| \|\Delta_{k+1}\| \\
&& + \alpha_{k} \int_{0}^{1} \|\nabla^{2} F(\xx^{(k)})  - \nabla^{2} F\big(\xx^{*} + t(\xx^{(k)} - \xx^{*})\big)\|_{\mathcal{K}} d t \|\Delta_{k}\| \|\Delta_{k+1}\| \\
&\leq& \left\| H(\xx^{(k)}) - \alpha_{k} \nabla^{2}F(\xx^{(k)}) \right\|_{\mathcal{K}} \|\Delta_{k}\| \|\Delta_{k+1}\| \\
&& + \alpha_{k} L \|\Delta_{k}\|^{2} \|\Delta_{k+1}\| \int_{0}^{1} (1-t) d t\\
&\leq& \left\| H(\xx^{(k)}) - \alpha_{k} \nabla^{2}F(\xx^{(k)}) \right\|_{\mathcal{K}} \|\Delta_{k}\| \|\Delta_{k+1}\| + \alpha_{k} \frac{L}{2} \|\Delta_{k}\|^{2} \|\Delta_{k+1}\|,
\end{eqnarray*}
where $\| A \|_{\mathcal{K}}$ is defined as in~\eqref{norm_cone_mat}.

On the other hand,we have
\begin{equation*}
\Delta_{k+1}^{T} H(\xx^{(k)}) \Delta_{k+1} \geq \lambda_{\min}^{\mathcal{K}} \left( H(\xx^{(k)}) \right) \|\Delta_{k+1}\|^{2}.
\end{equation*}
By Assumption~\eqref{H_strict_convex}, we have $\lambda_{\min}^{\mathcal{K}} \left( H(\xx^{(k)}) \right) > 0$. Hence we finally get
\begin{equation*}
\|\Delta_{k+1}\| \leq \frac{\left\| H(\xx^{(k)}) - \alpha_{k} \nabla^{2}F(\xx^{(k)}) \right\|_{\mathcal{K}}}{\lambda_{\min}^{\mathcal{K}} \left( H(\xx^{(k)}) \right)} \|\Delta_{k}\| + \alpha_{k} \frac{L}{ 2 \lambda_{\min}^{\mathcal{K}} \left( H(\xx^{(k)}) \right)}  \|\Delta_{k}\|^{2}.
\end{equation*}
\qed
\end{proof}

\comment Note that for the case of $H(\xx^{(k)}) = \nabla^{2} F(\xx^{(k)})$ and $\alpha_{k} = 1$, we exactly recover the convergence rate for the standard Newton's method~\cite[Proposition 1.4.1]{bertsekas1999nonlinear}.

Under the relaxed assumptions~\eqref{strong_convex_boundedness_opt}, we have the following result:
\begin{lemma}[Structural Lemma 2]
\label{structural_lemma_eig_relax}
Let Assumptions~\eqref{F_Lip} and~\eqref{strong_convex_boundedness_opt} hold. In addition, we assume that 
\begin{equation}
\pp^{T} H(\xx^{*}) \pp  > 0 , \quad \forall \; \pp \in \mathcal{K} \setminus \{0\},
\label{H_strict_convex_opt}
\end{equation}
and 
\begin{equation}
\|H(\xx) - H\big(\xx^{*}) \|_{\mathcal{K}} \leq \Gamma \|\xx - \xx^{*}\|, \quad \forall \xx \in \mathcal{D} \cap \mathcal{X}.
\label{H_Lip}
\end{equation}
Then, for the update~\eqref{structural_update}, if 
\begin{equation}
\|\xx^{(k)} - \xx^{*}\| \leq \frac{\lambda_{\min}^{\mathcal{K}} \left( H(\xx^{*}) \right)}{2\Gamma},
\label{initial_cond_gen}
\end{equation}
we have
\begin{equation*}
\|\xx^{(k+1)} - \xx^{*}\| \leq \rho_{0}  \|\xx^{(k)} - \xx^{*}\| +  \xi  \|\xx^{(k)} - \xx^{*}\|^{2},
\end{equation*}
where 
\begin{equation*}
\rho_{0} \defeq \frac{2 \left\| H(\xx^{*})  - \alpha_{k} \nabla^{2}F(\xx^{*})\right\|_{\mathcal{K}}}{\lambda_{\min}^{\mathcal{K}} \left( H(\xx^{*}) \right)}, \quad
\xi \defeq \frac{2\Gamma + \alpha_{k} L }{\lambda_{\min}^{\mathcal{K}} \left( H(\xx^{*}) \right)}.
\end{equation*}
\end{lemma}
\begin{proof}
As in Lemma~\ref{structural_lemma_eig}, let $\Delta_{k} = \|\xx^{(k)} - \xx^{*}\|$. We get,
\begin{eqnarray*}
\Delta_{k+1}^{T} H(\xx^{(k)}) \Delta_{k+1} &\leq& \Delta_{k+1}^{T} H(\xx^{(k)}) \Delta_{k} - \alpha_{k}\Delta_{k+1}^{T} \left(\int_{0}^{1} \nabla^{2} F\big(\xx^{*} + t(\xx^{(k)} - \xx^{*})\big) d t \right) \Delta_{k} \\
&=& \Delta_{k+1}^{T} \Big( H(\xx^{(k)}) - H(\xx^{*}) \Big) \Delta_{k} + \Delta_{k+1}^{T} \Big( H(\xx^{*}) - \alpha_{k} \nabla^{2}F(\xx^{*}) \Big) \Delta_{k} \\
&& +\alpha_{k}\Delta_{k+1}^{T} \Big( \int_{0}^{1} \nabla^{2}F(\xx^{*}) - \nabla^{2} F\big(\xx^{*} + t(\xx^{(k)} - \xx^{*})\big) d t \Big)\Delta_{k} \\
&\leq& \| H(\xx^{(k)}) - H(\xx^{*})\|_{\mathcal{K}} \|\Delta_{k}\| \|\Delta_{k+1}\| \\
&& + \left\| H(\xx^{*})  - \alpha_{k} \nabla^{2}F(\xx^{*})\right\|_{\mathcal{K}} \|\Delta_{k}\| \|\Delta_{k+1}\| \\
&& +\alpha_{k} \int_{0}^{1} \| \nabla^{2}F(\xx^{*}) - \nabla^{2} F\big(\xx^{*} + t(\xx^{(k)} - \xx^{*}) \|_{\mathcal{K}} d t \|\Delta_{k}\| \|\Delta_{k+1}\| \\
&\leq& \left\| H(\xx^{*})  - \alpha_{k} \nabla^{2}F(\xx^{*})\right\|_{\mathcal{K}} \|\Delta_{k}\| \|\Delta_{k+1}\|\\
&& + \left(\Gamma +\alpha_{k} L/2\right) \|\Delta_{k}\|^{2} \|\Delta_{k+1}\|.
\end{eqnarray*}

We also have
\begin{eqnarray*}
\Delta_{k+1}^{T} H(\xx^{(k)}) \Delta_{k+1} &=& \Delta_{k+1}^{T} H(\xx^{*}) \Delta_{k+1} + \Delta_{k+1}^{T} \Big(H(\xx^{(k)}) - H(\xx^{*}) \Big)\Delta_{k+1} \\
&\geq& \lambda_{\min}^{\mathcal{K}} \left(H(\xx^{*})\right) \|\Delta_{k+1}\|^{2} + \Delta_{k+1}^{T} \Big(H(\xx^{(k)}) - H(\xx^{*}) \Big)\Delta_{k+1} \\
&\geq& \lambda_{\min}^{\mathcal{K}} \left(H(\xx^{*})\right) \|\Delta_{k+1}\|^{2}  - \Gamma \|\Delta_{k}\| \|\Delta_{k+1}\|^{2} \\
&=& \Big(\lambda_{\min}^{\mathcal{K}} \left(H(\xx^{*})\right)  - \Gamma \|\Delta_{k}\|\Big) \|\Delta_{k+1}\|^{2} 
\end{eqnarray*}
By Assumption~\eqref{H_strict_convex_opt}, we have $\lambda_{\min}^{\mathcal{K}} \left(H(\xx^{*})\right) > 0$. Hence, by setting $\|\Delta_{k}\| \leq \lambda_{\min}^{\mathcal{K}} \left(H(\xx^{*})\right) /(2\Gamma)$, we get 
\begin{equation*}
\Big(\lambda_{\min}^{\mathcal{K}} \left(H(\xx^{*})\right) - \Gamma \|\Delta_{k}\|\Big) \geq \frac{\lambda_{\min}^{\mathcal{K}} \left(H(\xx^{*})\right)}{2},
\end{equation*}
and so
\begin{eqnarray*}
\|\Delta_{k+1}\| &\leq& \frac{\left\| H(\xx^{*})  - \alpha_{k} \nabla^{2}F(\xx^{*})\right\|_{\mathcal{K}}}{\Big(\lambda_{\min}^{\mathcal{K}} \left(H(\xx^{*})\right) - \Gamma \|\Delta_{k}\|\Big)}  \|\Delta_{k}\| + \frac{ (\Gamma + \alpha_{k} L/2)}{\Big(\lambda_{\min}^{\mathcal{K}} \left(H(\xx^{*})\right) - \Gamma \|\Delta_{k}\|\Big)}  \|\Delta_{k}\|^{2} \\
&\leq& \frac{2 \left\| H(\xx^{*})  - \alpha_{k} \nabla^{2}F(\xx^{*})\right\|_{\mathcal{K}}}{\lambda_{\min}^{\mathcal{K}} \left(H(\xx^{*})\right)}  \|\Delta_{k}\| + \frac{2\Gamma + \alpha_{k} L}{\lambda_{\min}^{\mathcal{K}} \left(H(\xx^{*})\right)}  \|\Delta_{k}\|^{2}.
\end{eqnarray*}
\qed
\end{proof}

\subsubsection{Main proofs}
\label{sec:main_proofs_hess}
Before delving in to the proofs, we first make a note regarding the step-size chosen for all of our algorithms. Let us denote
\begin{equation*}
A \defeq H(\xx^{(k)}) - \alpha_{k} \nabla^{2}F(\xx^{(k)}),
\end{equation*}
and note that
\begin{equation*}
\left\| A \right\|_{\mathcal{K}} = \max \left\{ \lambda_{\max}^{\mathcal{K}} \left( A \right), \lambda_{\max}^{\mathcal{K}} \left( -A \right) \right\}.
\end{equation*}
By Weyl's inequality~\cite[Theorem 8.4.11]{bernstein2009matrix}, we have
\begin{eqnarray*}
\lambda_{\max}^{\mathcal{K}} \left( A \right) &\leq& \lambda_{\max}^{\mathcal{K}} \left(H(\xx^{(k)})\right) - \alpha_{k} \lambda_{\min}^{\mathcal{K}} \left(\nabla^{2}F(\xx^{(k)})\right)  \\
\lambda_{\max}^{\mathcal{K}} \left( -A \right) &\leq& \alpha_{k} \lambda_{\max}^{\mathcal{K}} \left(\nabla^{2}F(\xx^{(k)})\right)  - \lambda_{\min}^{\mathcal{K}} \left(H(\xx^{(k)})\right).
\end{eqnarray*}
It is easy to see that setting
\begin{equation*}
\alpha_{k} = \frac{\lambda_{\max}^{\mathcal{K}} \left(H(\xx^{(k)})\right) + \lambda_{\min}^{\mathcal{K}} \left(H(\xx^{(k)})\right)}{\lambda_{\max}^{\mathcal{K}} \left(\nabla^{2}F(\xx^{(k)})\right) + \lambda_{\min}^{\mathcal{K}} \left(\nabla^{2}F(\xx^{(k)})\right)},
\end{equation*}
minimizes the maximum of both right hand sides. Now if the sampling is done as in Lemma~\ref{hoeffding_lemma_2}, then on the event that 
\begin{equation*}
\|H - \nabla^{2}F(\xx) \|_{\mathcal{K}} \leq \epsilon \gamma,
\end{equation*}
we must have
\begin{eqnarray*}
\lambda_{\min}^{\mathcal{K}} \left( H(\xx) \right) &\geq& \lambda_{\min}^{\mathcal{K}} \left(\nabla^{2}F(\xx^{(k)})\right) -\epsilon\gamma, \\
\lambda_{\max}^{\mathcal{K}} \left( H(\xx) \right) &\leq& \lambda_{\max}^{\mathcal{K}} \left(\nabla^{2}F(\xx^{(k)})\right) + \epsilon\gamma.
\end{eqnarray*}
Hence, the optimal step size for the linear term is $\alpha_{k} \approx 1$. As a result, for the following results, we always consider the ``natural'' Newton step size (i.e., $\alpha_{k} = 1$). 

In addition, using~\eqref{hoeffding_invertibility}, the requirement~\eqref{initial_cond_gen} holds if
\begin{equation}
\|\xx^{(k)} - \xx^{*}\| \leq \frac{(1-\epsilon)\gamma^{*}}{2L},
\label{initial_cond_spec}
\end{equation}
where we have $\Gamma = L$. Also, under Assumptions~\eqref{F_strong} and~\eqref{F_strong_opt}, $H(\xx^{(k)})$ and $H(\xx^{*})$ always satisfy~\eqref{H_strict_convex} and~\eqref{H_strict_convex_opt}, respectively.

\begin{proof}[Proof of Theorem~\ref{uniform_newton_convergence}]
Since $|\mathcal{S}|$ is set as described in Comment~\ref{sample_size_alg} and $\alpha_{k} = 1$, then~\eqref{linear_rate_factor} holds with probability $1-\delta$. Now the results follow immediately by applying Lemmas~\ref{structural_lemma_eig} and~\ref{structural_lemma_eig_relax} and noting that for our choice of $H$, we have $\Gamma = L$. \qed
\end{proof}

\begin{proof}[Proof of Theorem~\ref{uniform_newton_sufficient_cond}]
Using this particular choice of $\epsilon$, Theorem~\ref{uniform_newton_convergence}, for every $k$, yields
\begin{equation*}
\|\xx^{(k+1)} - \xx^{*}\| \leq \rho_{0} \|\xx^{(k)} - \xx^{*}\| + \xi \|\xx^{(k)} - \xx^{*}\|^{2}.
\end{equation*}
Now the result follows by requiring that $$\rho_{0} \|\xx^{(0)} - \xx^{*}\| + \xi \|\xx^{(0)} - \xx^{*}\|^{2} \leq \rho \|\xx^{(0)} - \xx^{*}\|.$$
Also note that for the latter part of Theorem~\ref{uniform_newton_sufficient_cond}, since $0 < \rho_{0} < \rho < 1$, by induction we have, for every $k$
\begin{equation*}
\|\xx^{(k)} - \xx^{*}\| \leq \frac{\rho - \rho_{0}}{\xi} \leq \frac{1}{\xi} = \frac{\gamma^{*} (1-\epsilon)}{3 L} \leq \frac{\gamma^{*} (1-\epsilon)}{2 L},
\end{equation*}
so the condition~\eqref{initial_cond_spec} is satisfied. 

Let $A_{k}$ denote the event that $\|\xx^{(k)} - \xx^{*}\| \leq \rho \|\xx^{(k-1)} - \xx^{*}\|$. The overall success probability is
\begin{equation*}
\Pr\Big(\bigcap_{k=1}^{k_{0}} A_{k}\Big) = \Pr\Big(A_{k_{0}} \Big| \bigcap_{k=1}^{k_{0}-1} A_{k} \Big) \Pr\Big(\bigcap_{k=1}^{k_{0}-1} A_{k}\Big) = \cdots = \prod_{k=1}^{k_{0}} \Pr\Big(A_{k} \Big| \bigcap_{i=1}^{k-1} A_{i} \Big) = (1-\delta)^{k_{0}},
\end{equation*}
since for every $k$, the conditional probability of a successful update $\xx^{(k+1)}$, given the past successful iterations $\{\xx_{i}\}_{i=1}^{k}$, is $1-\delta$.
\qed
\end{proof}

\begin{proof}[Proof of Theorem~\ref{uniform_newton_sufficient_cond_2}]
Theorem~\ref{uniform_newton_convergence}, for each $k$, gives
\begin{equation*}
\|\xx^{(k+1)} - \xx^{*}\| \leq \rho_{0}^{(k)} \|\xx^{(k)} - \xx^{*}\| + \xi^{(k)} \|\xx^{(k)} - \xx^{*}\|^{2},
\end{equation*}
where depending on the assumptions of the theorem, $\rho_{0}^{(k)}$ and $\xi^{(k)}$ are as in~\eqref{rho_xi} or~\eqref{rho_xi_opt} using $\epsilon^{(k)}$. Note that, by $\epsilon^{(k)} = \rho^{k} \epsilon$ and $\epsilon$ set as in Theorem~\ref{uniform_newton_sufficient_cond}, it follows that 
\begin{eqnarray*}
\rho_{0}^{(0)} &\leq& \rho_{0}, \\ 
\rho_{0}^{(k)} &\leq& \rho^{k} \rho_{0}, \\
\xi^{(k)} &\leq& \xi^{(k-1)}.
\end{eqnarray*}
We prove the result by induction on $k$. Define $\Delta_{k} \defeq \xx^{(k)} - \xx^{*}$. For $k=0$, by assumptions on $\rho$, $\rho_{0}$, and $\xi^{(0)}$, we have 
\begin{equation*}
\|\Delta_{1}\| \leq \rho_{0}^{(0)} \|\Delta_{0}\| + \xi^{(0)} \|\Delta_{0}\|^{2} \leq \rho_{0} \|\Delta_{0}\| + \xi^{(0)} \|\Delta_{0}\|^{2} \leq \rho \|\Delta_{0}\|. 
\end{equation*}
Now assume that~\eqref{sup_lin} holds up to the iteration $k$. For $k+1$, we get
\begin{eqnarray*}
\|\Delta_{k+1}\| &\leq& \rho_{0}^{(k)} \|\Delta_{k}\| + \xi^{(k)} \|\Delta_{k}\|^{2}\\
&\leq& \rho^{k} \rho_{0} \|\Delta_{k}\| + \xi^{(k)} \|\Delta_{k}\|^{2} \\
&\leq& \rho^{k} \rho_{0} \|\Delta_{k}\| + \xi^{(0)} \|\Delta_{k}\|^{2}.
\end{eqnarray*}
By induction hypothesis, we have $\|\Delta_{k-1}\| \leq \|\Delta_{0}\|$, and 
\begin{equation*}
\|\Delta_{k}\| \leq \rho^{k} \|\Delta_{k-1}\| < \frac{ \rho^{k} (\rho-\rho_{0})}{\xi^{(0)}},
\end{equation*}
hence,
\begin{equation*}
\|\Delta_{k+1}\| \leq \rho^{k+1} \|\Delta_{k}\|.
\end{equation*}

Under the local regularity Assumptions~\eqref{strong_convex_boundedness_opt}, we can see by induction that for every $k$ 
\begin{equation*}
\|\xx^{(k)} - \xx^{*}\| < \|\xx^{(0)} - \xx^{*}\| \leq \frac{\rho-\rho_{0}^{(0)}}{\xi^{(0)}} \leq \frac{\gamma^{*} \left(1-\epsilon^{(0)}\right)}{3 L} \leq \frac{\gamma^{*} \left(1-\epsilon^{(k)}\right)}{2 L},
\end{equation*}
so the condition~\eqref{initial_cond_spec} is satisfied. The overall success probability is computed as in the end of the proof of Theorem~\ref{uniform_newton_sufficient_cond}.\qed
\end{proof}

\begin{proof}[Proof of Theorem~\ref{uniform_newton_sufficient_cond_3}]
We only give the proof for the first part, as the proof for the second part is almost identical. By the choice of $\epsilon^{(k)}$ in~\eqref{rho_xi}, we have
\begin{equation*}
\rho_{0}^{(k)} =  \frac{1}{2 \ln(4+k)},
\end{equation*}
and as before
\begin{eqnarray*}
\rho_{0}^{(k)} &<& \rho_{0}^{(k-1)}, \\
\xi^{(k)} &\leq& \xi^{(k-1)}.
\end{eqnarray*}
We again prove the result by induction on $k$. For $k=0$, by assumptions on $\xi^{(0)}$, we have 
\begin{equation*}
\|\Delta_{1}\| \leq \rho_{0}^{(0)} \|\Delta_{0}\| + \xi^{(0)} \|\Delta_{0}\|^{2} = \frac{1}{2 \ln(4)} \|\Delta_{0}\| + \xi^{(0)} \|\Delta_{0}\|^{2} \leq \frac{1}{\ln(4)} \|\Delta_{0}\|. 
\end{equation*}
Now assume that~\eqref{sup_lin_slow} holds up to the iteration $k$. For $k+1$, we get
\begin{eqnarray*}
\|\Delta_{k+1}\| &\leq& \rho_{0}^{(k)} \|\Delta_{k}\| + \xi^{(k)} \|\Delta_{k}\|^{2}\\
&\leq& \frac{1}{2 \ln(4+k)} \|\Delta_{k}\| + \xi^{(0)} \|\Delta_{k}\|^{2}.
\end{eqnarray*}
Now consider 
\begin{equation*}
\phi(x) = \frac{\ln(4+x)}{\ln(3+x)}.
\end{equation*}
Since $\phi(0) < 2\ln(2)$ and 
\begin{equation*}
\frac{d \phi(x)}{dx} = \frac{(3+x)\ln(3+x)- (4+x)\ln(4+x)}{(4+x)(3+x) \ln^2(3+x)} < 0,\quad \forall x \geq 0,
\end{equation*}
i.e., $\phi(x)$ is decreasing, it follows that, we have
\begin{equation*}
\ln(4+k) \leq 2\ln(2) \ln(3+k), \quad \forall k \geq 0.
\end{equation*}
Since $\ln(3+k) \geq 1$, we get
\begin{equation*}
\left(1 + 2\ln(4+k)\right) \leq \ln(3+k) \left(1 + 4\ln(2)\right).
\end{equation*}
Now, by induction hypothesis, we have $\|\Delta_{k-1}\| \leq \|\Delta_{0}\|$, and 
\begin{equation*}
\|\Delta_{k}\| \leq \frac{1}{\ln(3+k)} \|\Delta_{k-1}\| < \frac{2 \gamma }{\ln(3+k) \left(1 + 2\ln(4)\right) L},
\end{equation*}
which using the above implies that
\begin{equation*}
\|\Delta_{k}\| \leq \frac{2 \gamma }{\left(1 + 2\ln(4+k)\right) L} = \frac{1}{2\ln(4+k) \xi^{(0)}}.
\end{equation*}
As a result, we get
\begin{equation*}
\|\Delta_{k+1}\| \leq \frac{1}{\ln(4+k)} \|\Delta_{k}\|.
\end{equation*}
\qed
\end{proof}

\begin{proof}[Proof of Theorem~\ref{spectral_hessian_convergence}]
Now suppose $|\mathcal{S}|$ is chosen as described in Comment~\ref{sample_size_alg}. By Lemma~\ref{hoeffding_lemma_2}, it follows that, with probability $1-\delta$, we have
\begin{equation*}
\lambda_{\min}^{\mathcal{K}} \left( H(\xx) \right) \geq (1-\epsilon) \gamma.
\end{equation*}
As a result, by construction, we get  $$\|\hat{H}(\xx) - H(\xx) \| = \max\left\{0,\lambda - \lambda_{\min}(H(\xx))\right\} \leq  \max\left\{0,\lambda - (1-\epsilon) \gamma\right\}.$$ Now, the result follows by writing
\begin{equation*}
\hat{H}(\xx^{(k)}) = \big(\hat{H}(\xx^{(k)}) - H(\xx^{(k)}) \big)  + H(\xx^{(k)}),
\end{equation*}
and using the same line of reasoning as in the proof of Lemma~\ref{structural_lemma_eig}.
\qed
\end{proof}
\begin{proof}[Proof of Theorem~\ref{spectral_hessian_convergence_relaxed}]
Proof goes along the same line as that of Lemma~\ref{structural_lemma_eig_relax}. In particular, using Assumption~\eqref{F_Lip_spectral}, we have the following chain of inequalities
\begin{eqnarray*}
\|\hat{H}(\xx^{(k)}) - \hat{H}(\xx^{*})\|_{2} &\leq& \|\hat{H}(\xx^{(k)}) - \hat{H}(\xx^{*})\|_{F} \leq \|H(\xx^{(k)}) - H(\xx^{*})\|_{F} \\
&\leq& \sqrt{p} \; \|H(\xx^{(k)}) - H(\xx^{*})\|_{2} \leq \sqrt{p} \; L \|\xx^{(k)} - \xx^{*}\|,
\end{eqnarray*}
where the second inequality follows by~\cite[Lemma VII.5.5]{bhatia2013matrix} and noting that the scalar version of the projection operator~\eqref{spectral_proj}, i.e., $\max(x,\lambda)$ is 1-Lipschitz. Specifically,
\begin{eqnarray*}
|\max(x,\lambda) - \max(y,\lambda)| &=& \hf \big|\lambda + x + |\lambda-x| - \lambda - y - |\lambda-y|\big| \\
&=& \hf \big|x - y + |\lambda-x| - |\lambda-y|\big| \\
&\leq& \hf|x - y| + \hf\big||\lambda-x| - |\lambda-y|\big| \\
&\leq& \hf|x - y| + \hf|x-y| \\
&\leq& |x-y|.
\end{eqnarray*} 
Now the result follows by writing
\begin{equation*}
\hat{H}(\xx^{(k)}) = \big(\hat{H}(\xx^{(k)}) - \hat{H}(\xx^{*}) \big) + \big(\hat{H}(\xx^{*}) - H(\xx^{*})\big) + H(\xx^{*}).
\end{equation*}
\qed
\end{proof}

\begin{proof}[Proof of Theorem~\ref{spectral_hessian_suff}]
By setting $\lambda^{(k)}$ as in~\eqref{spectral_threshold} and using Lemma~\ref{hoeffding_lemma_2}, we in fact have $\lambda^{(k)} \geq (1-\epsilon) \gamma$ and a requirement of Theorem~\ref{spectral_hessian_convergence} holds with probability $1-\delta$. Now we solve for $\epsilon^{(k)}$, so that for some $\rho_{0}^{(k)} > 0$, which is to be determined later, we have
\begin{equation*}
\rho_{0}^{(k)} \geq \frac{\lambda^{(k)} - (1-\epsilon^{(k)}) \gamma +  \gamma\epsilon^{(k)}}{\lambda^{(k)}}.
\end{equation*}
In other words, we need to choose 
\begin{equation*}
\epsilon^{(k)} \leq \frac{\gamma + (\rho_{0}^{(k)} - 1) \lambda^{(k)}} {2 \gamma}.
\end{equation*}
For this choice of $\epsilon^{(k)}$, using Theorem~\ref{spectral_hessian_convergence}, we get~\eqref{convergence_subsampled_hessian} with $\rho_{0}^{(k)}$ and $\xi^{(k)} = L/(2\lambda^{(k)})$. 
Now setting
\begin{eqnarray*}
\rho_{0}^{(k)} &=& 1 - \frac{2 \gamma}{3 \lambda^{(k)}} \\
\rho^{(k)} &=& 1 - \frac{\gamma}{2 \lambda^{(k)}},
\end{eqnarray*}
and noting that
\begin{equation*}
\frac{\rho^{(k)} - \rho_{0}^{(k)}}{\xi^{(k)}} = \frac{\gamma}{3 L},
\end{equation*}
it can be easily shown, by induction, that under Assumption~\eqref{initial_cond_spectral} and the choice of $\epsilon$ in~\eqref{epsilon_spectral}, we have
\begin{equation*}
\|\xx^{(k+1)} - \xx^{*}\| \leq \rho^{(k)} \|\xx^{(k)} - \xx^{*}\|
\end{equation*}
for every $k$. Furthermore, since two sampling steps in Algorithm~\ref{alg_spectral} are independent, the success probability for each iteration is $(1-\delta)^{2}$. Now the overall success probability can computed as in the end of the proof of Theorem~\ref{uniform_newton_sufficient_cond}.
\qed
\end{proof}

\begin{proof}[Proof of Theorem~\ref{spectral_hessian_suff_relaxed}]
Suppose at iteration $k$,~\eqref{initial_cond_spectral_relax} holds. By Weyl's inequality and Assumption~\eqref{F_Lip_spectral}, we have
\begin{equation*}
\left|\lambda_{\min}\left(\nabla^{2}F(\xx^{(k)})\right) - \lambda_{\min}\left(\nabla^{2}F(\xx^{*})\right)\right| \leq \| \nabla^{2}F(\xx^{(k)}) - \nabla^{2}F(\xx^{*}) \| \leq L \| \xx^{(k)} - \xx^{*} \|.
\end{equation*}
So by~\eqref{initial_cond_spectral_relax}, it follows that
\begin{eqnarray*}
\lambda_{\min}\left(\nabla^{2}F(\xx^{(k)})\right) &\geq& \lambda_{\min}\left(\nabla^{2}F(\xx^{*})\right) - L \| \xx^{(k)} - \xx^{*} \| \\
&\geq& \gamma^{*} - L \| \xx^{(k)} - \xx^{*} \|\\
&\geq& \frac{(6\sqrt{p} + 2) \gamma^{*}}{6 \sqrt{p} + 3}.
\end{eqnarray*}
Hence, by setting $\lambda^{(k)}$ as in~\eqref{spectral_threshold_relaxed}, we in fact have $\lambda^{(k)} \geq (1-\epsilon) \gamma^{*}$ and a requirement of Theorem~\ref{spectral_hessian_convergence_relaxed} holds with probability $1-\delta$. The rest of the proof follows the same reasoning as in the proof of Theorem~\ref{spectral_hessian_suff} by using the results of Theorem~\ref{spectral_hessian_convergence_relaxed}. \qed
\end{proof}

\begin{proof}[Proof of Theorem~\ref{ridge_hessian_convergence}]
The first part is straightforward by using Lemma~\ref{structural_lemma_eig}. The proof of the second part follows the same reasoning as in Lemma~\ref{structural_lemma_eig_relax} by noting
\begin{equation*}
\Delta_{k+1}^{T} \hat{H}(\xx^{(k)}) \Delta_{k+1} \geq \Big(\lambda + (1-\epsilon)\gamma^{*} - L \|\Delta_{k}\|\Big) \|\Delta_{k+1}\|^{2}.
\end{equation*}
Now~\eqref{initial_cond_spec} implies
\begin{equation*}
\lambda + (1-\epsilon)\gamma^{*} - L \|\Delta_{k}\| \geq  \lambda + \hf (1-\epsilon)\gamma^{*},
\end{equation*}
and the result follows. \qed
\end{proof}

\begin{proof}[Proof of Theorem~\ref{ridge_hessian_suff}]
Proof follows the same reasoning as before using Theorem~\ref{ridge_hessian_convergence}, so we omit the details. Note that, for the latter part of Theorem~\ref{ridge_hessian_suff}, since $1-\gamma^{*}/(\gamma^{*} + 2\lambda) < \rho_{0} < \rho < 1$, we have
\begin{equation*}
\frac{\rho - \rho_{0}}{\xi} \leq \frac{\gamma^{*}}{(\gamma^{*} + 2\lambda)\xi} = \frac{\gamma^{*}}{(\gamma^{*} + 2\lambda)} \frac{(1-\epsilon) \gamma^{*} + 2\lambda}{3L} \leq  \frac{\gamma^{*}}{3 L} \leq \frac{(1-\epsilon)\gamma^{*}}{2 L},
\end{equation*}
so the condition~\eqref{initial_cond_spec} is satisfied. The last inequality follows since, by the choice of $\epsilon$ and $\rho_{0} < 1$, we have
\begin{equation*}
\epsilon \leq \frac{\rho_{0} \gamma^{*} + 2(\rho_{0}  - 1) \lambda }{(2 + \rho_{0}) \gamma^{*}} \leq \frac{\rho_{0} }{2 + \rho_{0}} < \frac{1}{3}.
\end{equation*}
\qed
\end{proof}

\subsection{Proofs of the theorems and lemma of Section~\ref{sec:subsampl_newton_hess_grad}}
Here we give the proofs of all of the results in Section~\ref{sec:subsampl_newton_hess_grad}.

\begin{proof}[Proof of Lemma~\ref{randnla_lemma}]
Same as in the proof of Lemma~\ref{hoeffding_lemma}, let $U$ be an orthonormal basis for $\mathcal{K}$. Hence, by the assumption and the definition~\eqref{norm_cone_vec}, we have  that at a given $\xx \in \mathcal{D} \cap \mathcal{X}$
\begin{equation*}
\|U^{T} \nabla f_{i}(\xx) \| \leq G(\xx), \quad i =1,2,\ldots,n.
\end{equation*}
As mentioned before, the gradient $\nabla F(\xx)$ can be equivalently written as a product of two matrices as $\nabla F(\xx) = A B$, where
\begin{eqnarray*}
A &\defeq& \begin{pmatrix}
\mid & \mid & & \mid \\
\nabla f_{1}(\xx) & \nabla f_{2}(\xx) & \cdots & \nabla f_{n}(\xx)\\
\mid & \mid & & \mid \\
\end{pmatrix} \in \mathbb{R}^{p \times n}, \\
B &\defeq& \left(1/n,1/n,\ldots,1/n\right)^{T} \in \mathbb{R}^{n \times 1}.
\end{eqnarray*}
As a result, approximating the gradient using sub-sampling is equivalent to approximating the product $AB$ by sampling columns and rows of A and B, respectively, and forming matrices $\widehat{A}$ and $\widehat{B}$ such $\widehat{A}\widehat{B} \approx AB$. More precisely, for a random sampling index set $\mathcal{S}$, we can represent the sub-sampled gradient~\eqref{subsampled_G} by the product $\widehat{A} \widehat{B}$ where $\widehat{A} \in \mathbb{R}^{p \times |\mathcal{S}|}$ and $\widehat{B} \in \mathbb{R}^{|\mathcal{S}| \times 1}$ are formed by selecting uniformly at random and with replacement, $|\mathcal{S}|$ columns and rows of $A$ and $B$, respectively, rescaled by $\sqrt{n/|\mathcal{S}|}$. Now we can use~\cite[Lemma 11]{drineas2006fast} to get
\begin{equation*}
\|U^{T}AB - U^{T}\widehat{A}\widehat{B} \|_{F} = \|\nabla F(\xx) - \bgg(\xx) \|_{\mathcal{K}} \leq \frac{G(\xx)}{\sqrt{|\mathcal{S}|}} \Big( 1 + \sqrt{8 \ln \frac{1}{\delta}}\Big),
\end{equation*}
with probability $1-\delta$. Now the result follows by requiring that 
\begin{equation*}
\frac{G(\xx)}{\sqrt{|\mathcal{S}|}} \Big( 1 + \sqrt{8 \ln \frac{1}{\delta}}\Big) \leq \epsilon.
\end{equation*}
\qed
\end{proof}

\subsubsection{Structural Lemmas}
\label{sec:strutural_lemmas_hess_grad}
As before, we first give some structural lemmas which form the foundation of our main results for this Section.  
\begin{lemma}[Structural Lemma 3]
\label{structural_lemma_grad}
Let Assumptions~\eqref{F_Lip},~\eqref{strong_convex_boundedness} and~\eqref{H_strict_convex} hold. For the update~\eqref{structural_update_grad}, we have
\begin{equation*}
\|\xx^{(k+1)} - \xx^{*}\| \leq \eta + \rho_{0}\|\xx^{(k)} - \xx^{*}\| + \xi  \|\xx^{(k)} - \xx^{*}\|^{2},
\end{equation*}
where
\begin{align*}
&\eta \defeq \alpha_{k} \frac{\|\nabla F(\xx^{(k)}) - \bgg(\xx^{(k)})\|_{\mathcal{K}}}{\lambda_{\min}^{\mathcal{K}} \left( H(\xx^{(k)}) \right)},\\
&\rho_{0} \defeq \frac{\left\| H(\xx^{(k)}) - \alpha_{k} \nabla^{2}F(\xx^{(k)}) \right\|_{\mathcal{K}} }{\lambda_{\min}^{\mathcal{K}} \left( H(\xx^{(k)}) \right)}, \\
&\xi \defeq \frac{\alpha_{k} L}{2 \lambda_{\min}^{\mathcal{K}} \left( H(\xx^{(k)}) \right)}.
\end{align*}
where $\|\vv\|_{\mathcal{K}}$ is as in~\eqref{norm_cone_vec}.
\end{lemma}

\begin{proof}
The result is obtained as in the proof of Lemma~\ref{structural_lemma_eig}, and using the identity $\bgg(\xx^{(k)}) = \bgg(\xx^{(k)}) - \nabla F(\xx^{(k)}) + \nabla F(\xx^{(k)})$, and noting that $|\Delta_{k+1}^{T} \left( \bgg(\xx^{(k)}) - \nabla F(\xx^{(k)}) \right)| \leq \|\bgg(\xx^{(k)}) - \nabla F(\xx^{(k)})\|_{\mathcal{K}} \|\Delta_{k+1}\|$.
\qed
\end{proof}

As in Section~\ref{sec:strutural_lemmas_hess}, it is possible to have similar results as in Lemma~\ref{structural_lemma_grad} by replacing~\eqref{strong_convex_boundedness} with its local counterpart~\eqref{strong_convex_boundedness_opt}. The proof is omitted as it is similar to the proof of previous lemmas.
\begin{lemma}[Structural Lemma 4]
\label{structural_lemma_grad_relaxed}
Let Assumptions~\eqref{F_Lip},~\eqref{strong_convex_boundedness_opt},~\eqref{H_strict_convex_opt}, and~\eqref{H_Lip} hold. For the update~\eqref{structural_update_grad}, if $\xx^{(k)}$ satisfies~\eqref{initial_cond_gen}, we have
\begin{equation*}
\|\xx^{(k+1)} - \xx^{*}\| \leq \eta + \rho_{0}  \|\xx^{(k)} - \xx^{*}\| +  \xi  \|\xx^{(k)} - \xx^{*}\|^{2},
\end{equation*}
where 
\begin{align*}
&\eta \defeq 2 \alpha_{k} \frac{\|\nabla F(\xx^{(k)}) - \bgg(\xx^{(k)})\|_{\mathcal{K}}}{\lambda_{\min}^{\mathcal{K}} \left( H(\xx^{*}) \right)},\\
&\rho_{0} \defeq \frac{2 \left\| H(\xx^{*})  - \alpha_{k} \nabla^{2}F(\xx^{*})\right\|_{\mathcal{K}}}{\lambda_{\min}^{\mathcal{K}} \left( H(\xx^{*}) \right)}, \\
&\xi \defeq \frac{2\Gamma + \alpha_{k} L }{\lambda_{\min}^{\mathcal{K}} \left( H(\xx^{*}) \right)}.
\end{align*}
\end{lemma}

It is even possible to relax Assumption~\eqref{boundedness_opt} and give a more general result where the gradient approximation error is measured only at a local optimum $\xx^{*}$.

\begin{lemma}[Structural Lemma 5]
\label{structural_lemma_grad_relaxed_2}
Let Assumptions~\eqref{F_strong_opt},~\eqref{H_strict_convex_opt} and~\eqref{H_Lip} hold and let 
\begin{equation*}
J_{\bgg}(\xx) \defeq \frac{d \bgg(\xx)}{d \xx},
\end{equation*}
denote the Jacobian of $\bgg$ at $\xx$. Also assume that for some $T \geq 0$,
\begin{equation*}
\| H(\xx) - J_{\bgg}(\yy) \|_{\mathcal{K}} \leq T \|\xx - \yy \|, \quad \xx, \yy \in \mathcal{D} \cap \mathcal{X}, \text{ s.t. } \xx-\yy \in \mathcal{K}.
\end{equation*}
For the update~\eqref{structural_update_grad} with $\alpha_{k} = 1$, if $\xx^{(k)}$ satisfies~\eqref{initial_cond_gen}, we have
\begin{equation*}
\|\xx^{(k+1)} - \xx^{*}\| \leq \eta + \xi \|\xx^{(k)} - \xx^{*}\|^{2},
\end{equation*}
where
\begin{align*}
\eta &\defeq \frac{2 \|\nabla F(\xx^{*}) - \bgg(\xx^{*})\|_{\mathcal{K}}}{\lambda_{\min}^{\mathcal{K}} \left( H(\xx^{*}) \right)}, \\
\xi &\defeq \frac{T}{\lambda_{\min}^{\mathcal{K}} \left( H(\xx^{*}) \right)}.
\end{align*}
\end{lemma}

\begin{proof}
As in Lemma~\ref{structural_lemma_eig}, let $\Delta_{k} = \|\xx^{(k)} - \xx^{*}\|$. We get,
\begin{eqnarray*}
\Delta_{k+1}^{T} H(\xx^{(k)}) \Delta_{k+1} &\leq& \Delta_{k+1}^{T} H(\xx^{(k)}) \Delta_{k} - \Delta_{k+1}^{T} \bgg(\xx^{(k)})\\
&\leq& \Delta_{k+1}^{T} H(\xx^{(k)}) \Delta_{k} + \Delta_{k+1}^{T} \Big( \bgg(\xx^{*}) - \bgg(\xx^{(k)})\Big) \\
&& + \Delta_{k+1}^{T} \Big( \nabla F(\xx^{*}) - \bgg(\xx^{*}) \Big),
\end{eqnarray*}
where, for the last inequality, we used the fact that, by first order optimality condition, we have $\Delta_{k+1}^{T} \nabla F(\xx^{*}) \geq 0$. Using mean value theorem, we get
\begin{equation*}
\bgg(\xx^{(k)}) - \bgg(\xx^{*}) = \int_{0}^{1} J_{\bgg}(\xx^{*} + t (\xx^{(k)} - \xx^{*})) d t (\xx^{(k)} - \xx^{*}).
\end{equation*}
So it follows that
\begin{eqnarray*}
\Delta_{k+1}^{T} H(\xx^{(k)}) \Delta_{k+1} &\leq&  \Delta_{k+1}^{T}  \Big( \int_{0}^{1} H(\xx^{(k)}) - J_{\bgg}\left(\xx^{*} + t (\xx^{(k)} - \xx^{*})\right)   dt \Big) \Delta_{k} \\
&& + \Delta_{k+1}^{T} \Big( \nabla F(\xx^{*}) - \bgg(\xx^{*}) \Big) \\
&\leq& \int_{0}^{1} \| H(\xx^{(k)}) - J_{\bgg}\left(\xx^{*} + t (\xx^{(k)} - \xx^{*})\right) \|_{\mathcal{K}} dt \|\Delta_{k+1}\|\|\Delta_{k}\| \\
&& +  \|\nabla F(\xx^{*}) - \bgg(\xx^{*})\|_{\mathcal{K}} \|\Delta_{k+1}\| \\
&\leq& T \|\Delta_{k+1}\|\|\Delta_{k}\|^{2} \int_{0}^{1} (1-t) dt +  \|\nabla F(\xx^{*}) - \bgg(\xx^{*})\|_{\mathcal{K}} \|\Delta_{k+1}\|\\
&=& \frac{T}{2} \|\Delta_{k+1}\|\|\Delta_{k}\|^{2} +  \|\nabla F(\xx^{*}) - \bgg(\xx^{*})\|_{\mathcal{K}} \|\Delta_{k+1}\|.
\end{eqnarray*}
Now the same reasoning as in the end of the proof of Lemma~\ref{structural_lemma_eig_relax} gives the result. \qed
\end{proof}

\subsubsection{Main proofs}

\begin{proof}[Proof of Theorem~\ref{uniform_newton_convergence_grad}]
The results are immediately obtained using Lemmas~\ref{structural_lemma_grad} and~\ref{structural_lemma_grad_relaxed}. \qed
\end{proof}

\begin{proof}[Proof of Theorem~\ref{uniform_newton_convergence_grad_relax_2}]
First note that since $H(\xx^{(k)})$ and $\bgg(\xx^{(k)})$ are not independent, the joint probability that they are within the desired tolerance, by sub-additivity of probability, is lower bounded by $1 - 2\delta$. Now by construction, $J_{\bgg}(\yy) = H(\yy)$, and the result follows by Assumption~\eqref{F_Lip} on Lipschitz continuity of $\nabla^{2}f_{i}(\xx)$ and Lemma~\ref{structural_lemma_grad_relaxed_2} with $T = L$. \qed
\end{proof}

\begin{proof}[Proof of Theorem~\ref{uniform_newton_grad_sufficient_cond}]
Using Theorem~\ref{uniform_newton_convergence_grad}, the particular choice of $\epsilon_{1}$ and $\epsilon^{(k)}_{2} = \rho^{k} \epsilon_{2}$, for each $k$, gives
\begin{equation*}
\|\xx^{(k+1)} - \xx^{*}\| \leq \eta^{(k)} + \rho_{0} \|\xx^{(k)} - \xx^{*}\| + \xi \|\xx^{(k)} - \xx^{*}\|^{2},
\end{equation*}
where $\eta^{(k)} = \rho \eta^{(k-1)}$ and $\eta^{(0)} \leq \rho_{1} \sigma$. We prove the result by induction on $k$. Define $\Delta_{k} \defeq \xx^{(k)} - \xx^{*}$. For $k=0$, using Assumption~\eqref{initial_cond_grad} and noting that by the definition of $\sigma$
\begin{equation*}
\sigma = \frac{\rho - (\rho_{0} + \rho_{1})}{\xi},
\end{equation*}
we have 
\begin{equation*}
\|\Delta_{1}\| \leq \eta^{(0)}  + \rho_{0} \|\Delta_{0}\| + \xi \|\Delta_{0}\|^{2} \leq \rho_{1} \sigma + \rho_{0} \sigma + \xi \sigma^{2} = \rho \sigma.
\end{equation*}
Now assume that~\eqref{lin_grad} holds for $k$. For $k+1$, we get
\begin{eqnarray*}
\|\Delta_{k+1}\| &\leq& \eta^{(k)} + \rho_{0} \|\Delta_{k}\| + \xi \|\Delta_{k}\|^{2}\\
&=& \rho^{k} \eta^{(0)} + \rho_{0} \|\Delta_{k}\| + \xi \|\Delta_{k}\|^{2}\\
&\leq& \rho^{k} \rho_{1} \sigma + \rho_{0} \rho^{k} \sigma + \xi \rho^{2k} \sigma^{2} \quad \quad \text{(induction hypothesis)}\\
&=& \rho^{k} \left( \rho_{1} \sigma + \rho_{0} \sigma + \xi \rho^{k}  \sigma^{2} \right)\\
&\leq& \rho^{k} \left( \rho_{1} \sigma + \rho_{0} \sigma + \xi \sigma^{2} \right)  \quad \quad \text{(since $\rho < 1$)}\\
&=& \rho^{k+1} \sigma \quad \quad \text{(definition of $\sigma$)}.
\end{eqnarray*}
Under the local regularity assumptions~\eqref{strong_convex_boundedness_opt}, we can see by induction that for every $k$ 
\begin{equation*}
\|\xx^{(k)} - \xx^{*}\| \leq \sigma = \frac{\rho-(\rho_{0}+\rho_{1})}{\xi} \leq \frac{\gamma^{*} \left(1-\epsilon_{1}\right)}{2 L},
\end{equation*}
so the condition~\eqref{initial_cond_spec} is satisfied. The overall success probability is computed as in the end of the proof of Theorem~\ref{uniform_newton_sufficient_cond}.\qed
\end{proof}

\begin{proof}[Proof of Theorem~\ref{uniform_newton_grad_sufficient_cond_relax_2}]
First note that by~\eqref{epsilon_cond} and~\eqref{initial_cond_grad_2}, we get that
\begin{equation*}
\frac{\epsilon}{1-\epsilon} \leq \frac{\gamma^{*} \rho_{0} \sigma}{2}.
\end{equation*}
The proof is again by induction using Theorem~\ref{uniform_newton_convergence_grad_relax_2}. First note that, by construction and the value of $\epsilon$, 
\begin{equation*}
\eta^{(k)} = \frac{2 \epsilon^{(k)}}{\big(1-\epsilon^{(k)}\big) \gamma^{*} } = \frac{2 \rho^{k} \epsilon}{\big(1-\rho^{k} \epsilon\big) \gamma^{*} } \leq \frac{2 \rho^{k} \epsilon}{\big(1-\epsilon\big) \gamma^{*} } \leq \rho^{k} \rho_{0} \sigma,
\end{equation*}
and $\xi^{(k)} \leq \xi^{(k-1)}$. For $k=0$, we have
\begin{equation*}
\|\Delta_{1}\| \leq \eta^{(0)}  + \xi^{(0)} \|\Delta_{0}\|^{2} \leq \rho_{0} \sigma + \xi^{(0)} \sigma^{2} \leq \rho \sigma, 
\end{equation*}
where the last inequality follows from the definition~\eqref{initial_cond_grad_2} and noting that
$$\sigma = \frac{\rho - \rho_{0}}{2\xi^{(0)}}.$$

Now assume that~\eqref{lin_grad} holds for $k$. For $k+1$, we get
\begin{eqnarray*}
\|\Delta_{k+1}\| &\leq& \eta^{(k)} + \xi^{(k)} \|\Delta_{k}\|^{2}\\
&\leq& \rho^{k} \rho_{0} \sigma + \xi^{(0)} \rho^{2k} \sigma^{2} \quad \quad \text{(induction hypothesis)}\\
&=& \rho^{k} \left( \rho_{0} \sigma + \xi^{(0)} \rho^{k}  \sigma^{2} \right)\\
&\leq& \rho^{k} \left( \rho_{0} \sigma + \xi^{(0)}   \sigma^{2} \right) \quad \quad \text{(since $\rho < 1$)}\\
&\leq& \rho^{k+1} \sigma \quad \quad \text{(base case for $k=0$)}.
\end{eqnarray*}

Under the local regularity assumption~\eqref{F_strong_opt}, we can see by induction that for every $k$ 
\begin{equation*}
\|\xx^{(k)} - \xx^{*}\| \leq \sigma = \frac{\rho-\rho_{0}}{2\xi^{(0)}} \leq \frac{\gamma^{*} \left(1-\epsilon\right)}{2L},
\end{equation*}
so the condition~\eqref{initial_cond_spec} is satisfied. The overall success probability is computed as in the end of the proof of Theorem~\ref{uniform_newton_sufficient_cond}.
\qed
\end{proof}

\begin{proof}[Proof of Theorem~\ref{uniform_newton_grad_sufficient_cond_relax_3}]
First, as in the proof of Theorem~\ref{uniform_newton_grad_sufficient_cond_relax_2}, by~\eqref{epsilon_cond} and~\eqref{initial_cond_grad_2} , we get
\begin{equation*}
\frac{\epsilon}{1-\epsilon} \leq \frac{\gamma^{*} \rho_{0} \sigma}{2}.
\end{equation*}
Now we note that 
$\eta^{(0)} \leq \rho_{0} \sigma$, $\eta^{(k)} \leq \rho^{k} \eta^{(k-1)}$, $\xi^{(k)} \leq \xi^{(k-1)}$ and 
\begin{equation}
\tau^{(k)} = \rho^{k-1} \tau^{(k-1)} = \cdots = \rho^{k-1} \rho^{(k-2)} \cdots \rho^{1} \tau^{(1)} = \rho \prod_{i=1}^{k-1} \rho^{i}.
\label{tau_rate}
\end{equation}
Again, the result is obtained by induction on $k$. Define $\Delta_{k} \defeq \xx^{(k)} - \xx^{*}$. For the base case of $k=0$, we have 
\begin{eqnarray*}
\|\Delta_{1}\| &\leq& \eta^{(0)} + \xi^{(0)} \|\Delta_{0}\|^{2} \quad \text{(Theorem~\ref{uniform_newton_convergence_grad_relax_2} with the current parameters)} \\
&\leq& \rho_{0} \sigma + \xi^{(0)} \sigma^{2} \leq \rho \sigma = \tau^{(1)} \sigma,
\end{eqnarray*}
where for the inequality, we used the definition of $\sigma$ and noting that $$\sigma = \frac{\rho - \rho_{0}}{2\xi^{(0)}}.$$ Now assume that~\eqref{sup_lin_grad} holds for $k$. For $k+1$, we have
\begin{eqnarray*}
\|\Delta_{k+1}\| &\leq& \eta^{(k)}  + \xi^{(k)} \|\Delta_{k}\|^{2}  \quad \text{(Theorem~\ref{uniform_newton_convergence_grad_relax_2})} \\
&\leq& \rho^{k} \eta^{(k-1)} + \xi^{(k)} \|\Delta_{k}\|^{2} \\
&\leq& \rho^{k} \eta^{(k-1)} + \xi^{(0)} (\tau^{(k)})^{2} \sigma^{2} \quad \text{(since $\xi^{(k)} \leq \xi^{(0)}$ and inductive hypothesis)}\\
&\leq& \left(\prod_{i=1}^{k} \rho^{i}\right) \rho_{0}\sigma  + \left(\prod_{i=1}^{k-1} \rho^{ i }\right)^{2} \rho^{2} \xi^{(0)}  \sigma^{2} \quad \text{(definition of $\eta^{(k)}$ and~\eqref{tau_rate})}\\
&=& \left(\prod_{i=1}^{k} \rho^{i}\right) \left( \rho_{0}\sigma + \left(\prod_{i=1}^{k} \rho^{ i }\right) \frac{\rho^{2}}{\rho^{2k}}\xi^{(0)} \sigma^{2} \right) \\
&=& \left(\prod_{i=1}^{k} \rho^{i}\right) \left( \rho_{0}\sigma + \rho^{ \frac{k^{2} - 3k + 4}{2}  } \xi^{(0)} \sigma^{2} \right)\\
&\leq& \left(\prod_{i=1}^{k} \rho^{i}\right) \left( \rho_{0}\sigma + \xi^{(0)} \sigma^{2} \right)\\
&\leq& \left(\prod_{i=1}^{k} \rho^{i}\right) \rho \sigma\\
&\leq& \rho^{k} \left(\rho \prod_{i=1}^{k-1} \rho^{i}\right) \sigma\\
&\leq& \rho^{k} \tau^{(k)} \sigma\quad \text{(by~\eqref{tau_rate})} \\
&=& \tau ^{(k+1)} \sigma.
\end{eqnarray*}
\qed
\end{proof}

\end{document}